\documentclass[letterpaper,11pt,reqno]{amsart}
\usepackage{vmargin,color}
\usepackage{amsmath,amsfonts,amsthm,graphicx,amssymb}
\usepackage{esint}
\usepackage{caption}

\newcommand{\E}{\mathcal{E}}
\newcommand{\F}{\mathcal F}

\renewcommand{\H}{\mathcal{H}}
\renewcommand{\L}{\mathcal{L}}
\newcommand{\M}{\mathcal{M}}

\newcommand{\Sc}{\mathcal S}

\newcommand{\hd}{\mathrm{hd}\,}

\newcommand{\R}{\mathbb{R}}
\renewcommand{\SS}{\mathbb{S}}

\newcommand{\Om}{\Omega}

\renewcommand{\a}{\alpha}

\newcommand{\g}{\gamma}
\newcommand{\de}{\delta}
\newcommand{\e}{\varepsilon}
\renewcommand{\k}{\kappa}
\renewcommand{\l}{\lambda}
\newcommand{\s}{\sigma}

\newcommand{\om}{\omega}
\newcommand{\vphi}{\varphi}
\newcommand{\Lip}{{\rm Lip}}

\newcommand{\Div}{{\rm div}\,}
\newcommand{\Id}{{\rm Id}\,}
\newcommand{\dist}{{\rm dist}}

\newcommand{\spt}{{\rm spt}}

\newcommand{\weak}{\rightharpoonup}
\newcommand{\weakstar}{\stackrel{*}{\rightharpoonup}}

\newcommand{\pa}{\partial}

\newcommand{\cl}{\mathrm{cl}\,}

\newcommand{\p}{\mathbf{p}}

\newcommand{\C}{\mathcal{C}}

\newcommand{\KK}{\mathcal{K}}

\newcommand{\var}{\mathbf{var}\,}

\newcommand{\eps}{\varepsilon}

\newcommand{\cc}{\subset\joinrel\subset}
\newcommand{\ssubset}{\subset\joinrel\subset}

\newcommand{\Lt}{<<}
\newcommand{\mres}{\llcorner}

\theoremstyle{plain}
\newtheorem{theorem}{Theorem}[section]
\newtheorem{lemma}[theorem]{Lemma}
\newtheorem{corollary}[theorem]{Corollary}
\newtheorem{proposition}[theorem]{Proposition}
\newtheorem*{theorem*}{Theorem}
\newtheorem*{corollary*}{Corollary}

\theoremstyle{definition}

\newtheorem{remark}[theorem]{Remark}

\newtheorem*{notation*}{Notation}

\numberwithin{equation}{section}
\numberwithin{figure}{section}
\newcommand{\id}{{\rm id}\,}

\newcommand{\UM}{\mathcal{UM}\,}
\newcommand{\LM}{\mathcal{LM}\,}

\newcommand{\bY}{\mathbf{Y}}
\newcommand{\bT}{\mathbf{T}}

\newcommand{\Sing}{\mathrm{Sing}}
\newcommand{\bC}{\mathbf{C}}

\newcommand{\cS}{\mathcal{S}}

\title{Smoothness of collapsed regions \\ in a capillarity model for soap films}
\author{Darren King}
\address{Department of Mathematics, The University of Texas at Austin, 2515 Speedway, Stop C1200, Austin TX 78712-1202, United States of America}
\email{king@math.utexas.edu}
\author{Francesco Maggi}
\email{maggi@math.utexas.edu}
\author{Salvatore Stuvard}
\email{stuvard@math.utexas.edu}

\begin{document}

\begin{abstract} {\rm We study generalized minimizers in the soap film capillarity model introduced in \cite{maggiscardicchiostuvard,kms}. Collapsed regions of generalized minimizers are shown to be smooth outside of dimensionally small singular sets, which are thus empty in physical dimensions. Since generalized minimizers converge to Plateau's type surfaces in the vanishing volume limit, the fact that collapsed regions cannot exhibit $Y$-points and $T$-points (which are possibly present in the limit Plateau's surfaces) gives the first strong indication that singularities of the limit Plateau's surfaces should always be ``wetted'' by the bulky regions of the approximating generalized minimizers.}
\end{abstract}

\maketitle

\tableofcontents

\section{Introduction} \label{s:intro} \subsection{Overview} We continue the investigation of a model for soap films based on capillarity theory which was recently introduced by A. Scardicchio and the authors in \cite{maggiscardicchiostuvard,kms}. Soap films are usually modeled as minimal surfaces with a prescribed boundary: this idealization of soap films gives a model without length scales, which cannot capture those behaviors of soap films determined by their three-dimensional features, e.g. by their thickness. Regarding enclosed volume, rather than thickness, as a more basic geometric property of soap films, in \cite{maggiscardicchiostuvard,kms} we have started the study of soap films through capillarity theory, by proposing a {\bf soap film capillarity model} (see \eqref{def:SFCP} below). In this model, one looks for surface tension energy minimizers enclosing a fixed small volume, and satisfying a spanning condition with respect to a given wire frame. In \cite{kms} we have proved the existence of minimizers, and have shown their convergence to minimal surfaces satisfying the same spanning condition as the volume constraint converges to zero ({\bf minimal surfaces limit}). Although minimizers in the soap film capillarity model are described by regions of positive volume, these regions may fail to have uniformly positive thickness: indeed, in order to satisfy the spanning condition, minimizers may locally collapse onto surfaces. Understanding these collapsed surfaces, as well as their behavior in the minimal surfaces limit, is an important step in the study of the soap film capillarity model. In this paper we obtain a decisive progress on this problem, by showing the smoothness of collapsed surfaces, up to possible singular sets of codimension at least $7$. In particular, we show that, in physical dimensions, collapsed regions are smooth, thus providing strong evidence that, in the minimal surfaces limit, any singularities of solutions of the Plateau's problem should be ``wetted'' by the bulky parts of capillarity minimizers.

\subsection{The soap film capillarity model}\label{section recap} We start by recalling the formulation of our model for soap films hanging from a wire frame, together with the main results obtained in \cite{kms,kms2}. The wire frame is a compact set $W\subset\R^{n+1}$, $n\ge1$, and the region of space accessible to soap films is the open set
\[
\Om=\R^{n+1}\setminus W\,.
\]
A {\bf spanning class} in $\Omega$ is a non-empty family $\C$ of smooth embeddings $\gamma \colon \SS^1 \to \Omega$ which is homotopically closed \footnote{If $\gamma_0,\gamma_1$ are smooth embeddings $\SS^1\to \Omega$ with $\gamma_0 \in \C$ and $f \colon \left[0,1\right] \times \SS^1 \to \Omega$ is a continuous mapping such that $f(0,\cdot)=\gamma_0$ and $f(1,\cdot)=\gamma_1$ then also $\gamma_1 \in \C$.} in $\Omega$; correspondingly, a relatively closed subset $S$ of $\Omega$ is {\bf $\C$-spanning $W$} if $S\cap\g\ne\emptyset$ for every $\g\in\C$. Given choices of $W$ and $\C$, we obtain a formulation of {\bf Plateau's problem} (area minimization with a spanning condition) following Harrison and Pugh \cite{harrisonpughACV,harrisonpughGENMETH} (see also \cite{DLGM}), by setting
\begin{equation} \label{def:Plateau}
\ell = \ell(W,\C)= \inf\left\lbrace  \H^n (S) \, \colon \, S \in \cS \right\rbrace\,,
\end{equation}
where $\H^n$ denotes the $n$-dimensional Hausdorff measure on $\R^{n+1}$, and where
\begin{equation} \label{def:Plateau competitors}
\cS = \left\lbrace  S \subset \Om \, \colon \, \mbox{$S$ is relatively closed and $\C$-spanning $W$} \right\rbrace\,.
\end{equation}
Minimizers $S$ of $\ell$ exist as soon as $\ell<\infty$. They are, in the jargon of Geometric Measure Theory, {\bf Almgren minimal sets in $\Om$}, in the sense that they minimize area with respect to local Lipschitz deformations
\begin{equation}
  \label{almgren minimizer}
  \H^n(S\cap B_r(x))\le\H^n(f(S)\cap B_r(x))\,,
\end{equation}
whenever $f$ is a Lipschitz map with $\{f\ne\id\}\cc B_r(x)\cc \Om$ and $f(B_r(x))\subset B_r(x)$ (here $B_r(x)$ is the open ball of center $x$ and radius $r$ in $\R^{n+1}$). This minimality property is crucial in establishing that, in the physical dimensions $n=1,2$, minimizers of $\ell$ satisfy the celebrated {\bf Plateau's laws}, and are thus realistic models for actual soap films; see \cite{Almgren76,taylor76}, section \ref{section regularity of Almgren min}, and
\begin{figure}
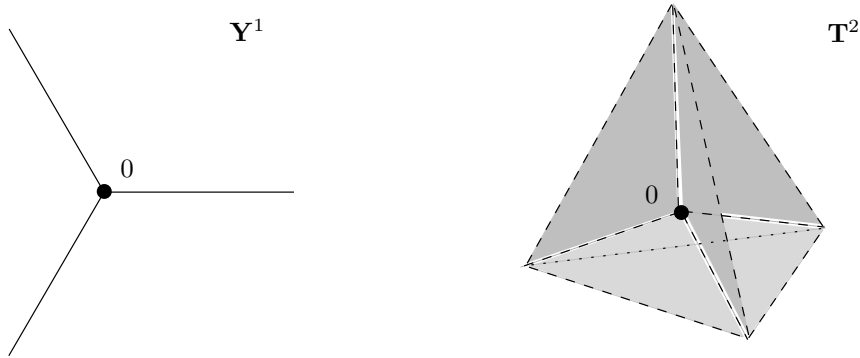
\caption{\small{When $n=1$, Almgren minimal sets are locally {\it isometric} either to lines or to $\bY^1\subset\R^2$, the cone with vertex at the origin spanned by $(1,0)$, $e^{i\,2\pi/3}$ and $e^{i\,4\pi/3}$. When $n=2$, Almgren minimal sets are locally diffeomorphic either to planes (and locally at these points they are smooth minimal surfaces), or to $\bY^1\times\R$, or to $\bT^2$, the cone with vertex at the origin spanned by the edges of a reference regular tetrahedron (for the purposes of this paper, there is no need to specify this reference choice).}}
\label{fig yt}
\end{figure}
Figure \ref{fig yt}.

\medskip

In capillarity theory (neglecting gravity and working for simplicity with a null adhesion coefficient) regions $E$ occupied by a liquid at equilibrium inside a container $\Om$ can be described by minimizing the area $\H^n(\Om\cap\pa E)$ of the boundary of $E$ lying inside the container while keeping the volume $|E|$ of the region fixed. When the fixed amount of volume $\e=|E|$ is small, minimizers in the capillarity problem take the form of small almost-spherical droplets sitting near the points of highest mean curvature of $\pa\Om$, see \cite{baylerosales,fall,maggimihaila}. To observe minimizers with a ``soap film geometry'', we impose the $\C$-spanning condition on $\Om\cap\pa E$, and come to formulate the {\bf soap film capillarity problem} $\psi(\eps) = \psi(\eps,W,\C)$, by setting
\begin{equation} \label{def:SFCP}
\psi(\e)= \inf\left\lbrace \H^n(\Om \cap \pa E)\, \colon \, \mbox{$E \in \E$, $|E| = \eps$, and $\Om \cap \pa E$ is $\C$-spanning $W$}\right\rbrace\,,
\end{equation}
where
\begin{equation} \label{def:SFCP competitors}
\E = \left\lbrace E \subset \Om \,\colon\, \mbox{$E$ is an open set and $\pa E$ is $\H^n$-rectifiable} \right\rbrace\,.
\end{equation}
Of course, a minimizing sequence $\{E_j\}_j$ for $\psi(\e)$ may find energetically convenient to locally ``collapse'' onto lower dimensional regions, see
\begin{figure}
  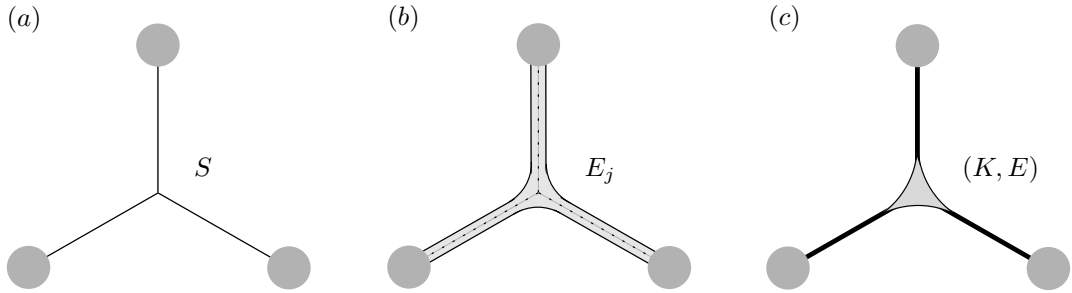
  \caption{{\small The soap film capillarity problem in the case when $W$ consists of three small disks centered at the vertexes of an equilateral triangle, and $\C$ is generated by three loops, one around each disk in $W$: (a) the unique minimizer $S$ of $\ell$ consists of three segments meeting at 120-degrees at a $Y$-point; (b) a minimizing sequence $\{E_j\}_j$ for $\psi(\e)$ will partly collapse along the segments forming $S$; (c) the resulting generalized minimizer $(K,E)$, where $K\setminus\pa E$ consists of three segments (whose area is weighted by $\F$ with multiplicity $2$, and which are depicted by bold lines), and where $E$ is a negatively curved curvilinear triangle enclosing a volume $\e$, and ``wetting'' the $Y$-point of $S$.}}
  \label{fig collapsing}
\end{figure}
Figure \ref{fig collapsing}. Hence, we do not expect to find minimizers of $\psi(\e)$ in $\E$, but rather to describe limits of minimizing sequences in the class
\begin{equation} \label{def:gen min class}
\begin{split}
\KK = \Big\{ (K,E) \, \colon \, & \mbox{$E \subset \Om$ is open with $\Om \cap \cl(\pa^*E) = \Om \cap \pa E \subset K$,} \\
& \mbox{$K \in \cS$ and $K$ is $\H^n$-rectifiable} \Big\}\,,
\end{split}
\end{equation}
(where $\pa^*E$ is the reduced boundary of $E$, and $\cl$ stands for topological closure in $\R^{n+1}$), and to compute the limit of their energies with the relaxed energy functional $\F$ defined on $\KK$ as
\begin{equation} \label{def:relaxed energy}
\F(K,E) = \H^n(\Om \cap \pa^*E) + 2\,\H^n(K \setminus \pa^*E) \qquad \mbox{for $(K,E) \in \KK$}\,.
\end{equation}
Notice the factor $2$ appearing as a weight for the area of $K\setminus\pa^*E$, due to the fact that $K\setminus\pa^*E$ originates as the limit of collapsing boundaries of $\Om\cap\pa E_j$. We can now recall the two main results proved in \cite{kms,kms2}, which state the existence of (generalized) minimizers of $\psi(\e)$ and prove the convergence of $\psi(\e)$ to the Plateau's problem $\ell$ when $\e\to 0^+$.

\begin{theorem}[Existence of generalized minimizers {\cite[Theorem 1.4]{kms}} and {\cite[Theorem 5.1]{kms2}}]\label{kms:existence}
Assume that $\ell=\ell(W,\C)<\infty$, $\Om$ has smooth boundary, and that
\begin{equation}
  \label{unica hp}
  \mbox{$\exists\,\tau_0>0$ such that $\R^{n+1} \setminus I_\tau(W)$ is connected for all $\tau < \tau_0$,}
\end{equation}
where $I_\tau(W)$ is the closed $\tau$-neighborhood of $W$.

\medskip

If $\e>0$ and $\{E_j\}_j$ is a minimizing sequence for $\psi(\e)$, then there exists $(K,E)\in\KK$ with $|E|=\e$ such that, up to possibly extracting subsequences, and up to possibly modifying each $E_j$ outside a large ball containing $W$ (with both operations resulting in defining a new minimizing sequence for $\psi(\e)$, still denoted by $\{E_j\}_j$), we have that,
  \begin{equation}\label{minimizing seq conv to gen minimiz}
    \begin{split}
    &\mbox{$E_j\to E$ in $L^1(\Om)$}\,,
    \\
    &\H^n\mres(\Om\cap\pa E_j)\weakstar \theta\,\H^n\mres K\qquad\mbox{as Radon measures in $\Om$}
    \end{split}
  \end{equation}
  as $j\to\infty$, where $\theta:K\to\R$ is upper semicontinuous and satisfies
  \begin{equation}
    \label{theta density}
    \mbox{$\theta= 2$ $\H^n$-a.e. on $K\setminus\pa^*E$},\qquad\mbox{$\theta=1$ on $\Om\cap\pa^*E$}\,.
  \end{equation}
  Moreover, $\psi(\e)=\F(K,E)$ and, for a suitable constant $C$, $\psi(\e)\le 2\,\ell+C\,\e^{n/(n+1)}$.
\end{theorem}

\begin{remark}
  {\rm Based on Theorem \ref{kms:existence}, we say that $(K,E)\in\KK$ is a {\bf generalized minimizer of $\psi(\e)$} if $|E|=\e$, $\F(K,E)=\psi(\e)$ and there exists a minimizing sequence $\{E_j\}_j$ of $\psi(\e)$ such that \eqref{minimizing seq conv to gen minimiz} and \eqref{theta density} hold.  }
\end{remark}

\begin{theorem}[Minimal surfaces limit, {\cite[Theorem 1.9]{kms}} and {\cite[Theorem 5.1]{kms2}}]\label{thm msl} Assume that $\ell=\ell(W,\C)<\infty$, $\Om$ has smooth boundary, and that \eqref{unica hp} holds. Then $\psi$ is lower semicontinuous on $(0,\infty)$ and $\psi(\e)\to 2\,\ell$ as $\e\to 0^+$. Moreover, if $\{(K_h,E_h)\}_h$ are generalized minimizers  of $\psi(\e_h)$ corresponding to $\e_h\to 0^+$ as $h\to\infty$, then there exists a minimizer $S$ of $\ell$ such that, up to extracting a subsequence in $h$, and as $h\to\infty$,
\[
2\,\H^n\mres(K_h\setminus \pa^*E_h)+\H^n\mres(\Om\cap\pa^*E_h)\weakstar 2\,\H^n\mres S\,,\qquad\mbox{as Radon measures in $\Om$}\,.
\]
\end{theorem}

Theorem \ref{kms:existence} and Theorem \ref{thm msl} open of course several questions on the properties of generalized minimizers at fixed $\e$, and on their behavior in the minimal surfaces limit $\e\to 0^+$. The two themes are very much intertwined, and in this paper we focus on the former, having in mind future developments on the latter. Before presenting our new results, we recall from \cite{kms} one of the most basic properties of generalized minimizers of $\psi(\e)$, namely, that they actually minimize the relaxed energy $\F$ among their (volume-preserving) diffeomorphic deformations. In particular, they satisfy a certain Euler-Lagrange equation which, by Allard's regularity theorem \cite{Allard}, implies a basic degree of regularity of $K$.

\begin{theorem}[{\cite[Theorem 1.6]{kms}} and {\cite[Theorem 5.1]{kms2}}]\label{thm basic regularity}
 Assume that $\ell=\ell(W,\C)<\infty$, $\Om$ has smooth boundary, and that \eqref{unica hp} holds. If $(K,E)$ is a generalized minimizer of $\psi(\e)$ and $f:\Om\to\Om$ is a diffeomorphism with $|f(E)|=|E|$, then
 \begin{equation}
   \label{minimality KE against diffeos}
   \F(K,E)\le\F(f(K),f(E))\,.
 \end{equation}
 In particular:

 \medskip

 \noindent (i) there exists $\l\in\R$ such that, for every $X\in C^1_c(\R^{n+1};\R^{n+1})$ with $X\cdot\nu_\Om=0$ on $\pa\Om$,
  \begin{equation}
    \label{stationary main}
    \l\,\int_{\pa^*E}X\cdot\nu_E\,d\H^n=\int_{\pa^*E}\Div^K\,X\,d\H^n+2\,\int_{K\setminus\pa^*E}\Div^K\,X\,d\H^n
  \end{equation}
  where $\Div^K$ denotes the tangential divergence operator along $K$;

  \medskip

  \noindent (ii) there exists $\Sigma\subset K$, closed and with empty interior in $K$, such that $K\setminus\Sigma$ is a smooth hypersurface, $K\setminus(\Sigma\cup\pa E)$ is a smooth embedded minimal hypersurface, $\H^n(\Sigma\setminus\pa E)=0$, $\Om\cap(\pa E\setminus\pa^*E)\subset \Sigma$ has empty interior in $K$, and $\Om\cap\pa^*E$ is a smooth embedded hypersurface with constant scalar mean curvature $\l$ (defined with respect to the outer unit normal $\nu_E$ of $E$).
\end{theorem}

\subsection{The exterior collapsed region of a generalized minimizer}\label{section collapsed} In \cite{kms2} we have started the study of the {\bf exterior collapsed region}
\[
K\setminus\cl(E)
\]
of a generalized minimizer $(K,E)$ of $\psi(\e)$. Indeed, the main result of \cite{kms2} is that if $K\setminus\cl(E)\ne\emptyset$, then the Lagrange multiplier $\l$ appearing in \eqref{stationary main} is non-positive, a fact that, in turn, implies the validity of the convex hull inclusion $K\subset{\rm conv}(W)$; see \cite[Theorem 2.8, Theorem 2.9]{kms2}. In this paper, we continue the study of $K\setminus\cl(E)$ by looking at its regularity. The basic fact that the multiplicity-one $n$-varifold defined by $K$ is stationary in $\Om\setminus\cl(E)$, see \eqref{stationary main}, implies the existence of a relatively closed subset $\Sigma$ of $K\setminus\cl(E)$ such that
\begin{equation}
  \label{sing and reg}
  \mbox{$K\setminus(\cl(E)\cup\Sigma)$ is a smooth minimal hypersurface}
\end{equation}
and $\H^n(\Sigma)=0$. The main result of this paper greatly improves this picture, by showing that $\Sigma$ is much smaller than $\H^n$-negligible.

\begin{theorem}[Sharp regularity for the exterior collapsed region] \label{t:main} Assume that $\ell=\ell(W,\C)<\infty$, $\Om$ has smooth boundary, and that \eqref{unica hp} holds. If $(K,E)$ is a generalized minimizer of $\psi(\e)$, then there exists a closed subset $\Sigma$ of $K\setminus\cl(E)$ such that $K\setminus(\Sigma\cup\cl(E))$ is a smooth minimal hypersurface,
\[
\Sigma=\emptyset\qquad\mbox{if $1\le n\le 6$}\,,
\]
$\Sigma$ is locally finite in $\Om\setminus\cl(E)$ if $n=7$, and $\Sigma$ is countably $(n-7)$-rectifiable (and thus has Hausdorff dimension $\le n-7$) if $n\ge 8$. In particular, in the physically relevant cases $n=1$ and $n=2$, the exterior collapsed region $K\setminus\cl(E)$ is a smooth stable minimal hypersurface in $\Om\setminus\cl(E)$.
\end{theorem}

\begin{remark}[Uniform local finiteness of the singular set]\label{remark locally finite}
  {\rm In fact, when $n\ge 7$ and $\Sigma$ is possibly non-empty, we will show that $\Sigma$ has locally finite $(n-7)$-dimensional Minkowski content (and thus locally finite $\H^{n-7}$-measure) in $\Om\setminus\cl(E)$; see section \ref{appendix NV local}.}
\end{remark}

\begin{remark}[Consequences for the minimal surfaces limit]
  \label{remark importante}
  {\rm A striking consequence of Theorem \ref{t:main} is that the exterior collapsed region $K\setminus\cl(E)$ is dramatically {\it more regular} than the generic minimizer of Plateau's problem $\ell$. For instance, in the physical dimension $n=2$, one can apply the work of Taylor \cite{taylor76}, as detailed for example in section \ref{section regularity of Almgren min}, to conclude that a minimizer $S$ of $\ell$ (which is known to be an Almgren minimal set in $\Om$ as defined in \eqref{almgren minimizer}) is locally diffeomorphic either to a plane (in which case, $S$ is locally a smooth minimal surface), or to the cone $\bY^1\times\R$ ($Y$-points), or to the cone $\bT^2$ ($T$-points); and, indeed these singularities are easily observable in soap films. At the same time, by Theorem \ref{t:main}, when $n=2$ the singular set of the exterior collapsed region is empty. Similarly, in arbitrary dimensions, the singular set of an $n$-dimensional minimizer $S$ of $\ell$ could have codimension one in $S$, while, by Theorem \ref{t:main}, the singular set of $K\setminus\cl(E)$ has {\it at least} codimension {\it seven} in $K\setminus\cl(E)$. This huge regularity mismatch between the exterior collapsed region and the typical minimizer in Plateau's problem has a second point of interest, as it provides strong evidence towards the conjecture that, in the minimal surfaces limit ``$(K_h,E_h)\to S$'' described in Theorem \ref{thm msl}, low codimension singularities of minimizers $S$ of $\ell$ may be contained (or even coincide, as it seems to be the case when $n=1$) with the set of accumulation points of the bulky regions $E_h$. This implication is of course not immediate, and will require further investigation.}
\end{remark}

\subsection{Outline of the proof of Theorem \ref{t:main}} The proof of Theorem \ref{t:main} is based on a mix of regularity theorems from Geometric Measure Theory, combined with two steps which critically hinge upon the specific structure of the variational problem $\psi(\e)$. A breakdown of the argument is as follows:

\medskip

\noindent {\bf Step one ($K\setminus\cl(E)$ is Almgren minimal in $\Om\setminus\cl(E)$):} In \eqref{minimality KE against diffeos} we have proved that $(K,E)$ minimizes $\F$ against diffeomorphic images which preserve the volume of $E$, an information which implies $\H^n(\Sigma)=0$ for the set $\Sigma$ in \eqref{sing and reg}. In this first step we greatly improve this information in the region away from $E$, by allowing for arbitrary Lipschitz deformations. Precisely, we show that $K\setminus\cl(E)$ is an Almgren minimal set in $\Om\setminus\cl(E)$, i.e.
\begin{equation} \label{e:almgren}
\H^n(K\cap B_r(x)) \leq \H^n(f(K)\cap B_r(x))
\end{equation}
for every Lipschitz map $f \colon \R^{n+1} \to \R^{n+1}$ such that $\{f \ne {\rm id}\} \subset B_r(x)$ and $f(B_r(x)) \subset B_r(x)$, with $B_r(x) \cc \Om \setminus \cl(E)$. Proving \eqref{e:almgren} is delicate, as discussed below.

\medskip

\noindent {\bf Step two ($K\setminus\cl(E)$ has no $Y$-points in $\Om\setminus\cl(E)$):} We construct ``wetting'' competitors, that cannot be realized as Lipschitz images of $K$, to rule out the existence of $Y$-points in $\Sigma$, that is points where $K$ is locally diffeomorphic to the cone $\bY^1\times\R^{n-1}$; see
\begin{figure}
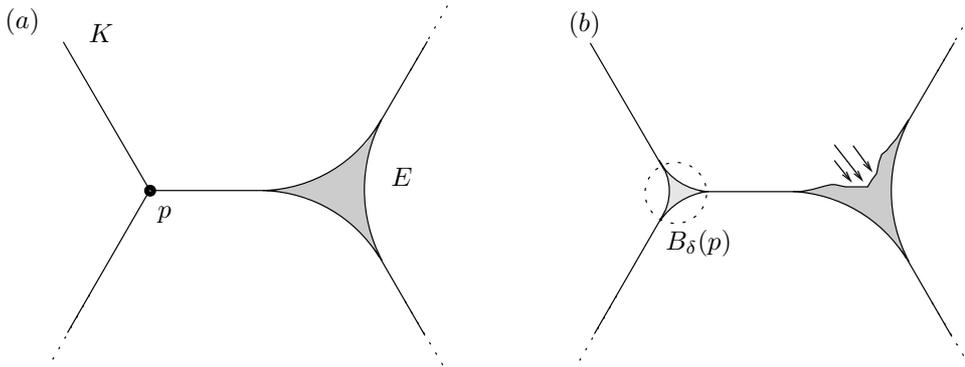
\caption{\small{Wetting competitors: (a) a local picture of a generalized minimizer $(K,E)$ when $n=1$, with a point $p$ of type $Y$; (b) the wetting competitor is obtained by first modifying $K$ at a scale $\de$ near $p$, so to save an ${\rm O}(\de)$ of length at the expense of an increase of ${\rm O}(\de^2)$ in area; the added area can be restored by pushing inwards $E$ at some point in $\pa^*E$, with a linear tradeoff between subtracted area and added length: in other words, to subtract an area of ${\rm O}(\de^2)$, we are increasing length by an ${\rm O}(\de^2)$ (whose size is proportional to the absolute value of the Lagrange multiplier $\l$ of $(K,E)$). If $\de$ is small enough in terms of $\lambda$, the ${\rm O}(\de)$ savings in length will eventually beat the ${\rm O}(\de^2)$ length increase used to restore the total area. In higher dimensions (where length and area become $\H^n$-measure and volume/Lebesgue measure respectively), wetting competitors are obtained by repeating this construction in the cylindrical geometry defined by the spine $\{0\}\times\R^{n-1}$ of $\bY^1\times\R^{n-1}$ near the $Y$-point $p$.}}
\label{fig wet}
\end{figure}
Figure \ref{fig wet}.

\medskip

\noindent {\bf Step three:} We combine the Almgren minimality of $K\setminus\cl(E)$ in $\Om\setminus\cl(E)$ and the absence of $Y$-points in $K\setminus\cl(E)$ with some regularity theorems by Taylor \cite{taylor76} and Simon \cite{Simon_cylindrical}, to conclude that the singular set $\Sigma$ of $K\setminus\cl(E)$ is $\H^{n-1}$-negligible. At the same time, \eqref{minimality KE against diffeos} implies that the multiplicity-one varifold associated to $K\setminus\cl(E)$ is not only stationary, but also stable in $\Om\setminus\cl(E)$. Therefore, we can exploit Wickramasekera's far reaching extension \cite{Wic} of a classical theorem of Schoen and Simon \cite{SchoenSimon81} to conclude that $\Sigma$ is empty if $1\le n\le 6$, is locally finite if $n=7$, and is $\H^{n-7+\eta}$-negligible for every $\eta>0$ if $n\ge 8$. This last information, combined with the Naber-Valtorta theorem \cite[Theorem 1.5]{NV_varifolds} implies that when $n\ge8$, $\Sigma$ is countably $(n-7)$-rectifiable, thus completing the proof of the theorem. We notice here that when $n=1,2$, one can implement this strategy by relying solely on Taylor's theorem \cite{taylor76}, thus avoiding the use of the Schoen-Simon-Wickramasekera theory, see Remark \ref{rmk reg if n12}.
Also, one can somehow rely on \cite{SchoenSimon81} only (rather than on the full strength of \cite{Wic}),
 see Remark \ref{rmk no wic}.
Finally, we notice that when $n\ge 8$, by further refining the above arguments, we can also show that $\Sigma$ is locally $\H^{n-7}$-finite: this is discussed in section \ref{appendix NV local}.

\medskip

We close this introduction by further discussing the construction of the competitors needed in carrying over step one of the above scheme. Indeed, this is a delicate point of the argument where we have made some non-obvious technical choices.

\medskip

\noindent {\bf Discussion of step one:} We illustrate the various aspects of the proof of \eqref{e:almgren} by means of
\begin{figure}
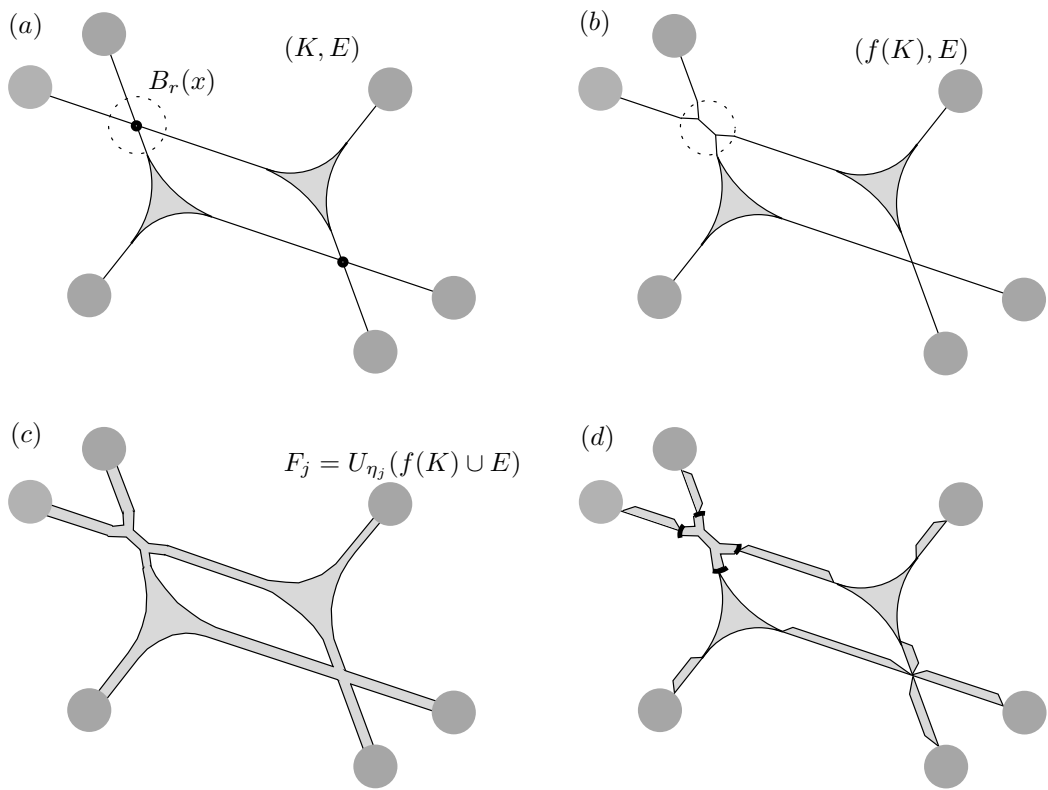
\caption{\small{Proving that the exterior collapsed region is an Almgren minimal set.}}
\label{fig stepone}
\end{figure}
Figure \ref{fig stepone}. In panel (a) we have a schematic representation of a generalized minimizer $(K,E)$ whose exterior collapsed region $K\setminus\cl(E)$ consists of various segments, intersecting along a singular set $\Sigma$ which is depicted by two black disks. We center $B_r(x)$ at one of the points in $\Sigma$, pick $r$ so that $B_r(x)$ is disjoint from $\cl(E)\cup W$, and in panel (b) we depict the effect on $K\cap B_r(x)$ of a typical area-decreasing Lipschitz deformation supported in $B_r(x)$ (notice that such a map is not injective, so \eqref{minimality KE against diffeos} is of no help here). As it turns out, one has $(f(K),E)\in\KK$: the only non-trivial point is showing that $f(K)$ is $\C$-spanning $W$, but this follows quite directly by arguing as in \cite[Proof of Theorem 4, Step 3]{DLGM}. Now, in order to deduce \eqref{e:almgren} from $\psi(\e)=\F(K,E)$ we need to find a sequence $\{F_j\}_j$ in the competition class of $\psi(\e)$ such that $\H^n(\Om\cap\pa F_j)\to \F(f(K),E)$ as $j\to\infty$. The obvious choice, at least in the situation depicted in Figure \ref{fig stepone}, would be taking \footnote{In the general situation, with $K\cap E$ possibly not empty, one should modify the formula for $F_j$ by removing $I_{\eta_j}(K\cap E)$. This fact is taken into account in the actual proof when we consider the set $A_1$ in Lemma \ref{l:one_sided_fattening} below.}
\[
F_j=U_{\eta_j}(f(K)\cup E)\,,
\]
for some $\eta_j\to 0^+$, where $U_\eta(S)$ denotes the open $\eta$-tubular neighborhood of the set $S$, see panel (c); for such a set $F_j$, we want to show that (i) $\H^n(B_r(x)\cap\pa F_j)\to 2\,\H^n(B_r(x)\cap f(K))$ as $j\to\infty$; and (ii) that $\Om\cap\pa F_j$ is $\C$-spanning $W$. Concerning problem (i), taking into account that
\[
\H^n(\Om\cap\pa F_\eta)\approx\frac{|U_\eta(f(K)\cup E)|}{\eta}\qquad\mbox{as $\eta\to 0^+$}\,,
\]
one wants first to show that $f(K)$ is Minkowski regular, in the sense that
\[
\lim_{\eta\to 0^+}\frac{|U_\eta(f(K))|}{2\eta}=\H^n(f(K))\,,
\]
and then to discuss the relation between $|U_\eta(f(K))|$ and $|U_\eta(f(K)\cup E)|$, which needs to keep track of those ``volume cancellations'' due to the parts of $U_\eta(f(K))$ which are contained in $E$. Discussing such cancellations is indeed possible through a careful adaptation of some recent works by Ambrosio, Colesanti and Villa \cite{ambrosiocolevilla,villa}. Addressing the Minkowski regularity of $f(K)$ requires instead the merging of two basic criteria for Minkowski regularity: ``Lipschitz images of compact subsets in $\R^n$ are Minkowski regular'' (Kneser's Theorem \cite{kneser}, see also \cite[3.2.28-29]{FedererBOOK} and \cite[Theorem 2.106]{AFP}) and ``compact $\H^n$-rectifiable sets with uniform density estimates are Minkowski regular'' (due to Ambrosio, Fusco and Pallara \cite[Theorem 2.104]{AFP}). In section \ref{s:Minkowski}, see in particular Theorem \ref{thm fine} below, we indeed merge these criteria by showing that ``Lipschitz images of compact $\H^n$-rectifiable sets with uniform density estimates are Minkowski regular''. (To apply this theorem to $f(K)$ we need of course to obtain uniform density estimates for $K$, which are discussed in section \ref{s:lde}, Theorem \ref{thm boundary density estimates}). We can thus come to a satisfactory solution of problem (i). However, we have not been able to solve problem (ii): in other words, it remains highly non-obvious if a set like $\Om\cap\pa F_j$ is always $\C$-spanning $W$, given the possibly subtle interactions between the geometries of $E$ and $f(K)$ and the operation of taking open neighborhoods. To overcome the spanning problem, we explore the possibility of defining $F_j$ as a {\it one-sided neighborhood} of $f(K)$ (which automatically contains the $\C$-spanning set $f(K)$ in its boundary), rather than as an open neighborhood of $f(K)$ (which contains the $\C$-spanning set $f(K)$ in its interior). As shown in \cite[Lemma 3.2]{kms2}, we can define $\C$-spanning one-sided neighborhoods of a pair $(K,E)\in\KK$ whenever $K$ is smoothly orientable outside of a meager closed subset of $K$. Thanks to Theorem \ref{thm basic regularity}-(ii) a generalized minimizer $(K,E)$ has enough regularity to define the required one-sided neighborhoods of $K$, but this regularity may be lost after applying the Lipschitz map $f$ to $K$: it thus seems that neither approach is going to work. The solution comes {\it by mixing the two methods}, as depicted in panel (d): inside $B_r(x)$, we define $F_j$ by taking an $\eta_j$-neighborhood of $f(K)$ -- which is fine, in terms of proving the $\C$-spanning condition, given the simple geometry of the ball and the care we will put in making sure that $\Om\cap\pa F_j$ contains the spherical subsets $\pa B_r(x)\cap U_{\eta_j}(f(K))$; inside $\Om\setminus\cl(B_r(x))$ we will define $F_j$ by the one-sided neighborhood construction -- notice that we have enough regularity in this region because $f(K)$ and $K$ coincide on $\Om\setminus\cl(B_r(x))$, We will actually need a variant of the one-sided neighborhood lemma \cite[Lemma 3.2]{kms2}, to guarantee that $\pa B_r(x)\cap U_{\eta_j}(f(K))$ is contained in $\Om\cap\pa F_j$, see Lemma \ref{l:one_sided_fattening}.

\medskip

\noindent {\bf Organization of the paper}: Section \ref{section notation} contains a summary of the notation used in the paper. In section \ref{s:Minkowski} we obtain the criterion for Minkowski regularity merging Kneser's theorem with \cite[Theorem 2.104]{AFP}, see Theorem \ref{thm fine}. In section \ref{s:lde} we discuss the lower density bounds up to the boundary wire frame needed to apply Theorem \ref{thm fine} to $K$, while in section \ref{s:almgren_minimal} we put together all these results to show the Almgren minimality of $K\setminus\cl(E)$ in $\Om\setminus\cl(E)$. In section \ref{s:noY} we construct the wetting competitors needed to exclude the presence of $Y$-points of $K\setminus\cl(E)$ in $\Om\setminus\cl(E)$, and finally, in section \ref{s:graph}, we illustrate the application of various regularity theorems \cite{taylor76,Simon_cylindrical,SchoenSimon81,Wic,NV_varifolds} needed to deduce Theorem \ref{t:main} from our variational analysis. Finally, in section \ref{appendix NV local}, we exploit more specifically the Naber-Valtorta results on the quantitative stratification of stationary varifolds and prove the local $\H^{n-7}$-estimates for $\Sigma$ mentioned in Remark \ref{remark locally finite}.

\medskip

\noindent {\bf Acknowledgment:} This work was supported by the NSF grants DMS 2000034, FRG-DMS 1854344, and RTG-DMS 1840314.

\section{Notation and terminology}\label{section notation} We summarize some basic definitions, mostly following \cite{SimonLN,maggiBOOK}.

\medskip

\noindent {\bf Radon measures and rectifiability:} We work in the Euclidean space $\R^{n+1}$ with $n \geq 1$. For $A \subset \R^{n+1}$, $\cl(A)$ denotes the topological closure of $A$ in $\R^{n+1}$, while $U_\eta(A)$ and $I_\eta(A)$ are the open and closed $\eta$-tubular neighborhoods of $A$, respectively. The open ball centered at $x \in \R^{n+1}$ with radius $r>0$ is denoted $B_r(x)$; given $1\leq k\leq n$ and a $k$-dimensional linear subspace $L \subset \R^{n+1}$, $B^L_r(x)$ denotes instead the open disc $B_r(x) \cap (x+L)$, and $B^k_r(x)$ is the corresponding shorthand notation when the subspace $L$ is clear from the context. We use the shorthand notation $B_r=B_r(0)$ and $B_r^k=B_r^k(0)$. If $A \subset \R^{n+1}$ is (Borel) measurable, then $|A|=\mathcal{L}^{n+1}(A)$ and $\H^s(A)$ denote its Lebesgue and $s$-dimensional Hausdorff measures, respectively, and we set $\om_k= \H^k(B^k_1)$. If $\mu$ is a Radon measure in $\R^{n+1}$, $A \subset \R^{n+1}$ is Borel, and $f \colon \R^{n+1} \to \R^d$ is continuous and proper, then $\mu \mres A$ and $f_\sharp \mu$ denote the restriction of $\mu$ to $A$ and the  push-forward of $\mu$ through $f$, respectively defined by $(\mu \mres A)(E) = \mu(A \cap E)$ for every Borel $E \subset \R^{n+1}$ and $(f_\sharp \mu)(F) = \mu(f^{-1}(F))$ for every Borel $F \subset \R^d$. The {\bf Hausdorff dimension} of $A$ is denoted $\dim_{\H}(A)$: it is the infimum of all real numbers $t \geq 0$ such that $\H^s(A) = 0$ for all $s>t$. Given an integer $1\leq k \leq n+1$, a Borel measurable set $M \subset \R^{n+1}$ is {\bf countably $k$-rectifiable} if it can be covered by countably many Lipschitz images of $\R^k$ up to a set of zero $\H^k$ measure; $M$ is {\bf (locally) $\H^k$-rectifiable} if it is countably $k$-rectifiable and, in addition, its $\H^k$ measure is (locally) finite. If $M$ is locally $\H^k$-rectifiable, then for $\H^k$-a.e. $x \in M$ there exists a unique $k$-dimensional linear subspace of $\R^{n+1}$, denoted $T_xM$, with the property that $\H^k \mres ((M-x)/r) \weakstar \H^k \mres T_xM$ in the sense of Radon measures in $\R^{n+1}$ as $r \to 0^+$: $T_xM$ is called the approximate tangent space to $M$ at $x$. If $f \colon \R^{n+1} \to \R^{n+1}$ is locally Lipschitz and $M$ is locally $\H^k$-rectifiable then the tangential gradient $\nabla^M f$ and the tangential jacobian $J^Mf$ are well defined at $\H^k$-a.e. point in $M$.

\medskip

\noindent {\bf Sets of finite perimeter:} A Borel set $E \subset \R^{n+1}$ is: of {\bf locally finite perimeter} if there exists an $\R^{n+1}$-valued Radon measure $\mu_E$  such that $\langle \mu_E, X \rangle = \int_E \Div(X)\, dx$ for all vector fields $X \in C^1_c(\R^{n+1};\R^{n+1})$; of {\bf finite perimeter} if, in addition, $P(E) = |\mu_E|(\R^{n+1})$ is finite. For any Borel set $F \subset \R^{n+1}$, the  relative perimeter of $E$ in $F$ is then defined by $P(E;F) = |\mu_E|(F)$. The {\bf reduced boundary} of a set $E$ of locally finite perimeter is the set $\pa^*E$ of all points $x \in \R^{n+1}$ such that the vectors $|\mu_E|(B_r(x))^{-1}\,\mu_E(B_r(x))$ converge, as $r\to 0^+$, to a vector $\nu_E(x) \in \mathbb{S}^n$: $\nu_E(x)$ is called the {\bf outer unit normal} to $\pa^*E$ at $x$. By De Giorgi's structure theorem, $\pa^*E$ is locally $\H^n$-rectifiable, with $\mu_E = \nu_E\,\H^n \mres \pa^*E$ and $|\mu_E| = \H^n \mres \pa^*E$.

\medskip

\noindent {\bf Integral varifolds}: An {\bf integral $n$-varifold} $V$ on an open set $U\subset\R^{n+1}$ is a continuous linear functional on $C^0_c(U\times G_n^{n+1})$ (where $G_n^{n+1}$ is the set of unoriented $n$-dimensional planes in $\R^{n+1}$) corresponding to a locally $\H^n$-rectifiable set $M$ in $U$, and a non-negative, integer valued function $\theta \in L^1_{{\rm loc}}(\H^n\mres M)$, so that
\[
V(\varphi)=\var(M,\theta)(\varphi) = \int_M \varphi(x,T_xM) \,\theta(x)\,d\H^n(x) \qquad \mbox{for all $\varphi \in C^0_c(U\times G_n^{n+1})$}\,.
\]
The function $\theta$, which is uniquely defined only $\H^n$-a.e. on $M$ is called the {\bf multiplicity} of $V$, while the Radon measure $\|V\|=\theta\,\H^n\mres M$ is the {\bf weight} of $V$ and $\spt\,V=\spt\,\|V\|$ is the {\bf support} of $V$. If $\Phi \colon U \to U'$ is a diffeomorphism, the {\bf push-forward of $V=\var(M,\theta)$ through $\Phi$} is the integral $n$-varifold $\Phi_\sharp V = \var(\Phi(M), \theta \circ \Phi^{-1})$ on $U'$. If $X\in C^1_c(U;\R^{n+1})$, then $\Div^TX=\vphi(x,T)$ defines a function $\vphi\in C^0_c(U\times G_n^{n+1})$: correspondingly, one says that $\vec{H} \in L^1_{{\rm loc}}(U;\R^{n+1})$ is the {\bf generalized mean curvature vector} of $V$ if
\begin{equation}
  \label{gen min curv vector}
  \int_M\,\theta\,\Div^MX\,d\H^n= \int_M X \cdot \vec{H} \, \theta\,d\H^n \qquad\forall X\in C^1_c(U;\R^{n+1})\,.
\end{equation}
When $\vec{H}=0$ we say that $V$ is {\bf stationary} in $U$: for example, if $M$ is a minimal hypersurface in $U$, then $V=\var(M,1)$ is stationary in $U$. Area monotonicity carries over from minimal surfaces to stationary varifolds, in the sense that the density ratios
\[
\frac{\|V\|(B_r(x))}{\om_n\,r^n}\quad\mbox{are increasing in $r\in\big(0,\dist(x,\pa U)\big)$}\,,
\]
with limit value as $r\to 0^+$ denoted by $\Theta_V(x)$ and called the {\bf density} of $V$ at $x$.

\section{Minkowski content of rectifiable sets} \label{s:Minkowski} The goal of this section is merging two well-known criteria for Minkowski regularity, Kneser's Theorem \cite{kneser} and \cite[Theorem 2.104]{AFP}, into Theorem \ref{thm fine} below. As explained in the introduction, this result will then play a crucial role in proving the Almgren minimality of exterior collapsed regions. It is convenient to introduce the following notation: given a compact set $Z\subset\R^d$ and an integer $k \in \{0,\ldots,d\}$, we define the {\bf upper and lower $k$-dimensional Minkowski contents} of $Z$ as
 \begin{eqnarray*}
  \UM^k(Z)&=&\limsup_{\eta\to 0^+}\frac{|U_\eta(Z)|}{\om_{d-k}\,\eta^{d-k}}\,,
  \\
  \LM^k(Z)&=&\liminf_{\eta\to 0^+}\frac{|U_\eta(Z)|}{\om_{d-k}\,\eta^{d-k}}\,.
\end{eqnarray*}
When $\UM^k(Z)=\LM^k(Z)$ we denote by $\M^k(Z)$ their common value, and call it the {\bf $k$-dimensional Minkowski content} of $Z$. If the $k$-dimensional Minkowski content of $Z$ exists, we say further that $Z$ is {\bf Minkowski $k$-regular} provided
\begin{equation} \label{d:Minkowski regularity}
\M^k(Z) \;=\; \H^k(Z)\,.
\end{equation}
It is easily seen that any $k$-dimensional $C^2$-surface with boundary in $\R^d$ is Minkowski $k$-regular, but, as said, more general criteria are available.

\begin{theorem}[Kneser's Theorem]\label{thm Federer 3229} If $Z\subset\R^k$ is compact and $f:\R^k\to\R^d$ is a Lipschitz map, then $f(Z)$ is Minkowski $k$-regular.
\end{theorem}

\begin{theorem}[{\cite[Theorem 2.106]{AFP}}]\label{thm AFP 2104} If $Z$ is a compact, countably $k$-rectifiable set in $\R^d$, and if there exists a Radon measure $\nu$ on $\R^d$ with $\nu\Lt\H^k$ and
  \[
  \nu(B_r(x))\ge c\,r^k\,,\qquad\forall x\in Z\,,\forall r<r_0\,,
  \]
  for positive constants $c$ and $r_0$, then $Z$ is Minkowski $k$-regular.
\end{theorem}

\begin{remark}
  {\rm For the reader's convenience we observe that Theorem \ref{thm AFP 2104} has the same statement as \cite[Theorem 2.104]{AFP}, although it should be noted that the existence of $\nu$ implies that $\H^k(Z)<\infty$, and thus that $Z$ is $\H^k$-rectifiable. For this reason we shall directly work with $\H^k$-rectifiable sets.}
\end{remark}

We now prove a result that mixes elements of both Theorem \ref{thm Federer 3229} and Theorem \ref{thm AFP 2104}, but that apparently does not follow immediately from them.

\begin{theorem}
  \label{thm fine}
  If $Z$ is a compact and $\H^k$-rectifiable set in $\R^d$ such that
  \begin{equation} \label{ldb}
  \H^k(Z\cap B_r(x))\ge c\,r^k\qquad\forall x\in Z\,,\forall r<r_0\,,
  \end{equation}
  and if $f:\R^d\to\R^d$ is a Lipschitz map, then $f(Z)$ is Minkowski $k$-regular.
\end{theorem}

We present a proof of Theorem \ref{thm fine} which follows the argument used in \cite{AFP} to prove Theorem \ref{thm Federer 3229}. We premise two propositions to the main argument.

\begin{proposition}
  \label{prop good case}
  If $Z$ is a compact and $\H^k$-rectifiable set in $\R^d$ such that
  \[
  \H^k(Z\cap B_r(x))\ge c\,r^k\qquad\forall x\in Z\,,\forall r<r_0\,,
  \]
  and if $f:\R^d\to\R^d$ is a Lipschitz map with $J^Z f>0$ a.e. on $Z$, then $f(Z')$ is Minkowski $k$-regular for any compact subset $Z' \subset Z$.
\end{proposition}

\begin{proof}
  Let
  \[
  \nu=f_\sharp\,(\H^k\mres Z)\,.
  \]
  If $y=f(x)\in f(Z)$ and $r\le \Lip(f)\,r_0$, then
  \[
  \nu(B_r(y))=\H^k(Z\cap f^{-1}(B_r(y)))\ge\H^k(Z\cap B_{r/\Lip(f)}(x))\ge \frac{c}{(\Lip (f))^k}\,r^k\,.
  \]
  Moreover $\nu\Lt \H^k$, since if $E\subset\R^d$ with $\H^k(E)=0$, then by $J^Z f>0$ on $Z$ we get
  \[
  \nu(E)=\H^k(Z\cap f^{-1}(E))=\int_{Z\cap f^{-1}(E)}\frac{J^Z f}{J^Z f}\,d\H^k=\int_{E\cap f(Z)}d\H^k(y)\int_{f^{-1}(y)\cap Z}\frac{d\H^0}{J^Zf}=0\,.
  \]
  We can thus apply Theorem \ref{thm AFP 2104} to $f(Z')$ for every $Z'\subset Z$ compact.
\end{proof}

\begin{proposition}
  \label{prop higher codim zeor jac}
  If $Z$ is a compact and $\H^k$-rectifiable set in $\R^d$ such that
  \[
  \H^k(Z\cap B_r(x))\ge c\,r^k\qquad\forall x\in Z\,,\forall r<r_0\,,
  \]
  and if $f:\R^d\to\R^d$ is a Lipschitz map, then
  \[
  \M^k(f(Z'))=0
  \]
  whenever $Z'\subset Z$ is compact with $J^Z f=0$ $\H^k$-a.e. on $Z'$.
\end{proposition}

\begin{proof}
  Let us define $f_\e:\R^d\to\R^d\times\R^d$ by setting
  \[
  f_\e(x)=\big(f(x),\e\,x\big)\,.
  \]
  If $x\in Z$ is such that $f$ is tangentially differentiable at $x$ along $Z$, then $f_\e$ is tangentially differentiable at $x$ along $Z$, and thus
  \[
  J^Z f_\e(x)>0\,,
  \]
  with $J^Z f_\e(x)\to J^Z f(x)$ as $\e\to 0^+$ and $J^Z f_\e\le(\Lip (f))^k+1$ for $\e<\e_0$. By Proposition \ref{prop good case} and the area formula, since $f_\e$ is injective we get
  \[
  \M^k(f_\e(Z'))=\H^k(f_\e(Z'))=\int_{Z'} J^Z f_\e\to 0\qquad\mbox{as $\e\to 0^+$}\,,
  \]
  where in computing the limit we have used $J^Z f=0$ $\H^k$-a.e. on $Z'$. Thus,
  \[
  \lim_{\e\to 0^+}
  \M^k(f_\e(Z'))=0\,.
  \]
  At the same time if $\eta>0$, then
  \[
  \Big\{(y,z)\in\R^d\times\R^d:y\in B_\eta(f(x))\,, z\in B_\eta(\e\,x)\,, x\in Z'\Big\}\subset U_{2\,\eta}(f_\e(Z'))\,,
  \]
  so that Fubini's theorem gives
  \[
  |U_\eta(f(Z'))|\,\om_d\,\eta^d\le |U_{2\eta}(f_\e(Z'))|\,.
  \]
  Dividing by $\eta^{2d-k}$ we get
  \[
  \UM^k(f(Z'))\,\le C(d,k)\,\limsup_{\eta\to 0^+}\frac{|U_{2\eta}(f_\e(Z'))|}{\eta^{2d-k}}\le C(d,k)\,\M^k(f_\e(Z'))\,,
  \]
  and letting $\e\to 0^+$ we conclude the proof.
\end{proof}

\begin{proof}[Proof of Theorem \ref{thm fine}] For brevity, set $S = f(Z)$. The rectifiability of $S$ gives $\LM^k(S)\ge\H^k(S)$, see e.g. \cite[Proposition 2.101]{AFP}, so that we only need to prove $\UM^k(S)\le\H^k(S)$. We set $F=\{J^Z f>0\}$ (that is $f$ is tangentially differentiable along $Z$ with positive tangential Jacobian on $F$) and pick $Z_0\subset \{J^Z f=0\}\subset Z\setminus F$ compact with the property that
\begin{equation} \label{error1}
\H^k\big(Z\setminus (F\cup Z_0)\big)<\s\,,
\end{equation}
for some $\s>0$. In this way, by Proposition \ref{prop higher codim zeor jac}, we have
\begin{equation} \label{zero piece}
\M^k(S_0)=0\qquad\mbox{where $S_0=f(Z_0)$}\,.
\end{equation}
Since $S$ is compact and $\H^k$-rectifiable we can find a countable disjoint family $\{S_i\}_i$ of compact subsets of $S$, which covers $S$ modulo $\H^k$, and such that $S_i=f_i(Z_i)$ for injective Lipschitz maps $f_i$ with uniformly positive Jacobian on $\R^k$. By Theorem \ref{thm Federer 3229},
\begin{equation} \label{good pieces}
\M^k(S_i)=\H^k(S_i) \qquad \mbox{for every $i$}\,.
\end{equation}
We pick $N$ so that
\begin{equation}
  \label{Hk delta}
  \H^k\Big(S\setminus \bigcup_{i=1}^N S_i\Big)<\de\,,
\end{equation}
for $\de$ to be chosen depending on $\s$, and set
\begin{equation} \label{error2}
S^*=S\setminus \Big(S_0\cup\bigcup_{i=1}^N S_i\Big)\,.
\end{equation}
Next, we further distinguish points in $S^*$ depending on their distance from $\bigcup_{i=0}^N S_i$. More precisely, for any arbitrary $\lambda \in \left( 0,1 \right)$ we define the compact set
\begin{equation} \label{to cover}
S^{**}=S\setminus U_{\l\eta}\Big(\bigcup_{i=0}^N S_i\Big)\,,
\end{equation}
and then we apply the Besicovitch covering theorem to cover
\begin{equation} \label{covering}
S^{**} \subset \bigcup_{j \in J} B_{\l\,\eta}(y_j)\,,
\end{equation}
where $J$ is a finite set of indexes, each $y_j \in S^{**}$, and each point of $S^{**}$ belongs to at most $\xi(d)$ distinct balls in the covering. Notice that
\begin{eqnarray*}
  \bigcup_{j\in J}B_{\l\,\eta}(y_j)&\subset&U_{\l\,\eta}(S)\setminus\bigcup_{i=0}^N S_i\,,
  \\
  S\cap\bigcup_{j\in J}B_{\l\,\eta}(y_j)&\subset& S^*\,.
\end{eqnarray*}
Furthermore, $B_{\l\,\eta/\Lip\,f}(x_j)\subset f^{-1}(B_{\l\,\eta}(y_j))$ for some $x_j \in Z$ such that $y_j=f(x_j)$. The lower density bound \eqref{ldb} then yields
\begin{eqnarray*}
  \#(J)\,\frac{c\,(\l\,\eta)^k}{(\Lip f)^k}&\le&\sum_{j\in J}\H^k\Big(Z\cap B_{\l\,\eta/\Lip f}(x_j)\Big)
  \\
  &\le&\sum_{j\in J}\H^k\Big(Z\cap f^{-1}(B_{\l\eta}(y_j))\Big)
  \\
  &\le&\xi(d)\,\H^k\Big(Z\cap f^{-1}\Big(\bigcup_{j\in J}B_{\l\eta}(y_j)\Big)\Big)
    \\
  &\le&\xi(d)\,\H^k(Z\cap f^{-1}(S^*))\,.
\end{eqnarray*}
However,
\begin{eqnarray*}
  \H^k(Z\cap f^{-1}(S^*))&=&\H^k(Z\cap f^{-1}(S^*)\cap F)+\H^k(Z\cap f^{-1}(S^*)\setminus F)
  \\
  &\le&
  \nu(S^*)+\H^k(Z\setminus(F\cup Z_0))\le  \nu(S^*)+\s\,,
\end{eqnarray*}
provided we set
\[
\nu=f_\sharp[\H^k\mres (Z\cap F)]\,.
\]
Since $J^Z f > 0$ on $F$, we have, for any Borel set $A \subset \R^d$
\[
\nu(A)=\int_{Z\cap F\cap f^{-1}(A)}\frac{J^Z f}{J^Z f}=\int_{A}d\H^k(y)\int_{f^{-1}(y)\cap Z\cap F}\frac{d\H^0}{J^Zf}
\]
so that $\nu\Lt\H^k$. Therefore a suitable choice of $\de=\de(\s)$ in \eqref{Hk delta} gives
\[
\nu(S^*)<\s\,,
\]
and we have thus proved that
\begin{equation} \label{number of balls}
\#(J) \leq C(d,c,\Lip(f))\, \s\, \lambda^{-k} \, \eta^{-k}\,.
\end{equation}
We can now conclude the argument. From the definition of $S^*$ it follows that
\begin{equation}
S = \bigcup_{i=0}^N S_i \cup S^*\,,
\end{equation}
and thus that
\begin{equation} \label{tb_nbd}
U_\eta(S) \subset \bigcup_{i=0}^N U_\eta(S_i) \cup U_\eta(S^*) \qquad \mbox{for all $\eta > 0$}\,.
\end{equation}
On the other hand, by the definition of $S^{**}$
\begin{equation} \label{key}
U_\eta(S^*) \subset U_\eta(S^{**}) \cup \bigcup_{i=0}^N U_{(1+\lambda)\,\eta} (S_i)\,,
\end{equation}
which, together with \eqref{covering},  gives
\begin{equation} \label{final_tb_nbd}
U_\eta(S) \subset \bigcup_{i=0}^N U_{(1+\l)\,\eta}(S_i) \cup \bigcup_{j \in J} B_{(1+\l)\,\eta}(y_j)\,.
\end{equation}
By means of \eqref{number of balls} we achieve
\begin{equation}
|U_\eta(S)| \leq \sum_{i=0}^N |U_{(1+\l)\,\eta}(S_i)| + C(d,c,\Lip(f)) \, \s \, \lambda^{-k}\, (1+\lambda)^d \, \eta^{d-k}\,,
\end{equation}
so that, dividing by $\omega_{d-k}\,\eta^{d-k}$, taking the limit as $\eta \to 0^+$, and using \eqref{zero piece} and \eqref{good pieces} we obtain
\begin{equation} \label{final}
\begin{split}
\UM^k (S) &\leq (1+\l)^{d-k} \sum_{i=1}^N  \H^k(S_i) + C(d,k,c,\Lip(f)) \, \s \, \lambda^{-k} \, (1+\lambda)^{d}\\
&\leq (1+\lambda)^{d-k} \, \H^k(S) + C(d,k,c,\Lip(f)) \, \s \, \lambda^{-k} \, (1+\lambda)^{d}\,.
\end{split}
\end{equation}
The conclusion follows by letting first $\s \to 0^+$ and then $\l \to 0^+$.
\end{proof}

We close this section by proving a useful localization statement.

\begin{proposition}[Localization of Minkowski content] \label{prop_localization}
   If $Z$ is a compact and $\H^k$-rectifiable set in $\R^d$ such that
  \[
  \M^k(Z)=\H^k(Z)\,,
  \]
  then
  \[
  \lim_{\eta\to 0^+}\frac{|U_\eta(Z)\cap E|}{\om_{d-k}\eta^{d-k}}=\H^k(Z\cap E)
  \]
  whenever $E$ is a Borel set with $\H^k(K\cap\pa E)=0$.
\end{proposition}

\begin{proof}
  If we set
  \[
  \mu_\eta=\frac{\L^d\mres U_\eta(Z)}{\omega_{d-k}\eta^{d-k}}\,,\qquad \mu=\H^k\mres Z\,,
  \]
  then we just need to prove that, as $\eta\to 0^+$, $\mu_\eta\weakstar \mu$ in $\R^d$. To this end, we first consider an open set $A$, set $A_\eta=\{x\in A:\dist(x,\pa A)\ge\eta\}$, and notice that, for $\eta<\eta_0$,
  \begin{eqnarray*}
    \mu_\eta(A)&=&\frac{|U_\eta(Z)\cap A|}{\om_{d-k}\eta^{d-k}}\ge
    \frac{|U_\eta(Z\cap A_\eta)|}{\om_{d-k}\eta^{d-k}}
    \\
    &\ge&
    \frac{|U_\eta(Z\cap A_{\eta_0})|}{\om_{d-k}\eta^{d-k}}
  \end{eqnarray*}
  so that, by \cite[Proposition 2.101]{AFP} and since $Z\cap A_{\eta_0}$ is compact and $\H^k$-rectifiable
  \[
  \liminf_{\eta\to 0^+}\mu_\eta(A)\ge\H^k(Z\cap A_{\eta_0})\,.
  \]
  Letting $\eta_0\to 0^+$ we get
  \begin{equation}
    \label{prop loc lsc}
      \liminf_{\eta\to 0^+}\mu_\eta(A)\ge\mu(A)\qquad\mbox{$\forall A\subset\R^d$ open}\,.
  \end{equation}
  Since $\M^k(Z)=\H^k(Z)$ means that $\mu_\eta(\R^d)\to\mu(\R^d)$ as $\eta\to 0^+$, we find that, for every compact set $H\subset\R^d$,  \begin{eqnarray} \label{prop loc usc}
    \mu(H)=\mu(\R^d)-\mu(A)\ge\lim_{\eta\to 0^+}\mu_\eta(\R^d)-\liminf_{\eta\to 0^+}\mu_\eta(A)
    \ge\limsup_{\eta\to 0^+}\mu_\eta(H)\,,
  \end{eqnarray}
  where we have used \eqref{prop loc lsc} with $A=\R^d\setminus H$. By a standard criterion for weak-star convergence of Radon measures, \eqref{prop loc lsc} and \eqref{prop loc usc} imply that $\mu_\eta\weakstar \mu$ in $\R^d$ as $\eta\to 0^+$.
\end{proof}

\section{Uniform lower density estimates} \label{s:lde} The application of Theorem \ref{thm fine} to $K$ (where $(K,E)$ is a generalized minimizer of $\psi(\e)$), requires proving uniform lower density estimates for $\cl(K)$ (recall that $K$ is compact relatively to $\Om$, not to $\R^{n+1}$). Now, it is a consequence of the analysis carried out in \cite{kms} that there exists a radius $r_* > 0$ such that
\begin{equation} \label{interior_dlb}
\H^n(K\cap B_r(x))\ge \om_n\,r^n\,,\qquad\forall x\in K\,, r < r_*\,,B_r(x)\cc\Om\,.
\end{equation}
However, the lower density estimate in \eqref{interior_dlb} is not sufficient to apply Theorem \ref{thm fine}, because its radius of validity degenerates as $x$ approaches $\cl (K) \setminus K = \cl (K) \cap \pa \Om$. We thus need an improvement, which is provided in the following theorem.

\begin{theorem}[Uniform lower density estimates]
  \label{thm boundary density estimates}
  If $(K,E)$ is a generalized minimizer of $\psi(\e)=\psi(\e,W,\C)$, then there exist $c=c(n)>0$ and $r_0=r_0(n,W,|\l|)>0$ such that
  \begin{equation} \label{e:uniform_lde}
  \H^n(K\cap B_r(x))\ge c\,r^n\qquad\forall x\in\cl(K)\,, \forall r<r_{0}\,.
  \end{equation}
  Here $\l$ is the Lagrange multiplier of $(K,E)$, as introduced in Theorem \ref{thm basic regularity}-(i).
\end{theorem}

The proof of Theorem \ref{thm boundary density estimates} starts with the remark that the integral $n$-varifold $V$ naturally associated to $(K,E)$ has generalized mean curvature in $L^\infty$ and it satisfies a distributional formulation of Young's law. More precisely, letting $V$ be the $n$-varifold $V = \var(K,\theta)$ defined by $K$ with multiplicity function
\[
\theta(x) =
\begin{cases}
1 &\mbox{if $x \in \pa^*E$}\\
2 & \mbox{if $x \in K \setminus \pa^*E$}\,,
\end{cases}
\]
then
\[
\|V\| = \H^n \mres (\Om \cap \pa^*E) + 2\,\H^n \mres (K \setminus \pa^*E)
\]
and, by considering \eqref{stationary main} in Theorem \ref{thm basic regularity} on vector fields compactly supported in $\Om$, we see that $V$ has generalized mean curvature vector $\vec{H} = \lambda\, 1_{\pa^*E} \, \nu_E$ in $\Om$. Actually, \eqref{stationary main} says more, since it allows for vector fields not necessarily supported in $\Om$, provided they are tangential to $\pa\Om$, i.e. \eqref{stationary main} gives
\begin{equation} \label{first variation}
\int \Div^KX \, d\|V\| = \int X \cdot \vec{H} \, d\|V\| \quad \mbox{$\forall\,X \in C^1(\Omega; \R^{n+1})$ with $X \cdot \nu_\Om = 0$ on $\pa \Om$}\,.
\end{equation}
The extra information conveyed in \eqref{first variation} is that, in a distributional sense, $V$ has contact angle $\pi/2$ with $\pa \Om$. The consequences of the validity of \eqref{first variation} have been extensively studied in the classical work of Gr\"uter and Jost \cite{gruterjost}, and their work has been recently extended to arbitrary contact angles by Kagaya and Tonegawa \cite{kagayatone}. In particular, if $s_0 \in \left(0,\infty \right)$ is such that the tubular neighborhood $U_{s_0}(\pa W)$ admits a well-defined nearest point projection map $\Pi \colon U_{s_0}(\pa W) \to \pa W$ of class $C^1$ then \cite[Theorem 3.2]{kagayatone} ensures the existence of a constant $C = C(n,s_0)$ such that for any $x \in U_{s_0/6}(\pa W) \cap {\rm cl}(\Om)$ the map
\begin{equation} \label{monotonicity_formula}
r \in (0,s_0/6)\mapsto \frac{\|V\|(B_r(x)) + \|V\|(\tilde B_r(x))}{\omega_n \, r^n} \,e^{(|\l|+C)r}
\end{equation}
is increasing, where
\begin{equation} \label{reflection}
\tilde B_r(x) = \left\lbrace y \in \R^{n+1} \, \colon \, \tilde y \in B_r(x) \right\rbrace \,, \qquad \tilde y = \Pi(y) + (\Pi(y) - y)
\end{equation}
denotes a sort of nonlinear reflection of $B_r(x)$ across $\pa W$. In particular, for $x$ as above the limit
\begin{equation} \label{boundary_density}
\sigma(x) = \lim_{r\to 0^+} \frac{\|V\|(B_r(x)) + \|V\|(\tilde B_r(x))}{\omega_n \, r^n}
\end{equation}
exists for every $x \in U_{s_0/6}(\pa W) \cap \cl(\Om)$, and the map $x \mapsto \sigma(x)$ is upper semicontinuous in there; see \cite[Corollary 5.1]{kagayatone}. The uniform density estimate \eqref{e:uniform_lde} will be deduced as a consequence of the above monotonicity reault, together with the following simple geometric lemma:
\begin{lemma}\label{lemma_ovvio}
Suppose that $x \in U_{s_0}(\pa W)$, and $\rho > 0$ is such that $\dist(x,\pa W) \leq \rho$ and $B_{\rho}(x) \subset U_{s_0}(\pa W)$. Then:
\begin{equation} \label{ovvio2}
\tilde B_{\rho}(x) \subset B_{5\rho}(x)\,.
\end{equation}
\end{lemma}

\begin{proof}
  [Proof of Lemma \ref{lemma_ovvio}] See \cite[Lemma 4.2]{kagayatone}.
\end{proof}

\begin{proof}[Proof of Theorem \ref{thm boundary density estimates}]
First observe that \eqref{e:uniform_lde} holds with $c=\omega_n$ for all $x \in K \setminus U_{s_0/6}(\pa W)$ as soon as $r < \min\{r_*,s_0/6\}$. Therefore, we can assume that
\begin{equation} \label{close_to_boundary}
x \in \cl(K) \cap U_{s_0/6}(\pa W)\,.
\end{equation}
Also note that for points as in \eqref{close_to_boundary} it holds $\sigma(x) \geq 1$: by upper semicontinuity of $\sigma$ on $U_{s_0/6}(\pa W) \cap \cl(\Om)$, we just need to show this when, in addition to \eqref{close_to_boundary}, we have $x \in K$, and indeed in this case
\[
\sigma(x) \geq \lim_{r\to 0^+} \frac{\H^n(K\cap B_r(x))}{\omega_n\, r^n} \geq 1
\]
thanks to \eqref{interior_dlb}. Now we fix $r < r_0 = \min\{r_*,5s_0/6\}$, and distinguish two cases depending on the validity of
\begin{equation} \label{no cross}
\dist(x,\pa W) > \frac{r}{5}\,.
\end{equation}
If \eqref{no cross} holds, then by \eqref{interior_dlb}
\[
\H^n(K \cap B_r(x)) \geq \H^n(K \cap B_{r/5}(x)) \geq \omega_n \, \left( \frac{r}{5} \right)^n\,,
\]
so that \eqref{e:uniform_lde} holds. If $\dist(x,\pa W) \leq r/5$, then, thanks to the obvious inclusion $B_{r/5}(x) \subset U_{s_0}(\pa W)$ we can apply Lemma \ref{lemma_ovvio} with $\rho = r/5$, and \eqref{ovvio2} yields $\tilde B_{r/5}(x) \subset B_r(x)$. Hence, by exploiting $\sigma(x) \geq 1$ and \eqref{monotonicity_formula} we get
\[
\begin{split}
c_n\,r^n &\leq \sigma(x)\,\omega_n\,\left(\frac{r}{5}\right)^n \\
&\leq \left( \|V\|(B_{r/5}(x)) + \|V\|(\tilde B_{r/5}(x))   \right) \, e^{(|\lambda|+C)\,r/5} \\
&\leq 2\, \|V\|(B_r(x)) \, e^{(|\lambda|+C)\,r_0} \leq 8\,\H^n(K\cap B_r(x))\,,
\end{split}
\]
up to further decreasing $r_0$.
\end{proof}

\begin{corollary} \label{cor:Minkowski final for soap films}
Let $(K,E)$ be a generalized minimizer of $\psi(\eps)$, and let $f \colon \R^{n+1} \to \R^{n+1}$ be Lipschitz. Then
\begin{equation} \label{e:Minkowski final}
\lim_{\eta \to 0^+} \frac{|U_\eta(f(\cl(K)))  \cap E |}{2\eta} = \H^n(f(\cl(K)) \cap E)
\end{equation}
whenever $E$ is a Borel set with $\H^n(f(\cl(K))  \cap \pa E )=0$.
\end{corollary}

\begin{proof}
  Immediate from Theorem \ref{thm fine}, Proposition \ref{prop_localization} and Theorem \ref{thm boundary density estimates}.
\end{proof}

\section{Minimality with respect to Lipschitz deformations} \label{s:almgren_minimal} In this section we complete the first step of our strategy, by proving the Almgren minimality of the exterior collapsed set.

\begin{theorem}\label{proposition Lipschitz minimizing}
  If $(K,E)$ is a generalized minimizer of $\psi(\e)$, $B_r(x)\cc\Om\setminus\cl(E)$, and  $f:\R^{n+1}\to\R^{n+1}$ is a Lipschitz map with $\{f\ne\id\}\subset B_r(x)$ and $f(B_r(x)) \subset B_r(x)$, then $\H^n(K\cap B_r(x))\le\H^n(f(K)\cap B_r(x))$.
\end{theorem}

As explained in the introduction, an important tool in the proof is the construction of one-sided neighborhoods of $K$. This point is discussed in the following lemma, which is, in fact, an extension of \cite[Lemma 3.2]{kms2} (which corresponds to the case $U=\emptyset$).
\begin{figure}
    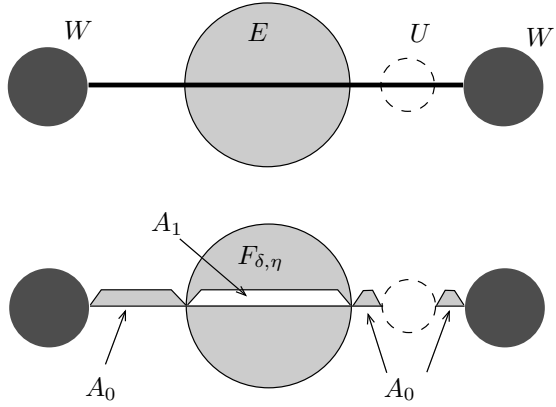\caption{{\small The construction in Lemma \ref{l:one_sided_fattening} gives a one-sided neighborhood of $K$ away from $\cl(E)\cup\cl(U)\cup W$, thus defining an open set $F_{\de,\eta}$ which contains $K\setminus \cl(U)$ in its boundary, and which collapses onto $K\setminus \cl(U)$ as $\eta\to 0^+$.}}\label{fig a0a1U}
  \end{figure}

\begin{lemma} \label{l:one_sided_fattening}
Let $K \subset \Om$ be a relatively compact and $\H^n$-rectifiable set, let $E \subset \Om$ be an open set with $\Om \cap \cl(\pa^*E) = \Om \cap \pa E \subset K$, and let $U \ssubset \Om \setminus \cl (E)$ be an open set. Suppose that $\Sigma \subset K$ is a closed subset with empty interior relatively to $K$ such that $K \setminus \Sigma$ is a smooth hypersurface in $\Om$ such that there exists $\nu \in C^\infty(K \setminus \Sigma ; \mathbb{S}^n)$ with $\nu(x)^\perp = T_x(K \setminus \Sigma)$ at every $x \in K \setminus \Sigma$. Set
\[
M=K \setminus (\Sigma \cup \pa E \cup \cl(U))\,,
\]
and decompose $M=M_0\cup M_1$ by letting
\begin{equation*}
M_0 = (K \setminus \Sigma) \setminus (\cl (E) \cup \cl (U))\,, \quad M_1 = (K \setminus \Sigma) \cap E\,.
\end{equation*}
Let $\|A_M\|(x)$ be the maximal principal curvature (in absolute value) of $M$ at $x$. For given $\eta,\delta \in \left( 0, 1 \right)$, define a positive function $u \colon M \to \left(0,\eta\right]$ by setting
\begin{equation} \label{osf:the function u}
u(x) = \min\left\lbrace \eta\,,\,   \frac{\dist(x,\Sigma \cup \pa E \cup \cl(U) \cup W)}{2}\,,\, \frac{\delta}{\|A_M\|(x)}    \right\rbrace\,,
\end{equation}
and let
\begin{eqnarray*}
A_0 &=& \left\lbrace  x + t\,u(x)\,\nu(x)\,\colon\, x \in M_0\,,\, 0 < t < 1  \right\rbrace\,,\\
A_1 &=& \left\lbrace  x + t\,u(x)\,\nu(x)\,\colon\, x \in M_1\,,\, 0 < t < 1  \right\rbrace\,,\\
F_{\delta,\eta} &=& A_0 \cup (E \setminus \cl (A_1))\,.
\end{eqnarray*}
Then, $F_{\de,\eta} \subset \Om \setminus \cl(U)$ is open, $\pa F_{\de,\eta}$ is $\H^n$-rectifiable, and
\begin{eqnarray} \label{osf1}
 K \setminus \cl (U) & \subset & \Om \cap \pa F_{\delta,\eta}  \setminus \cl(U)\,, \\ \label{osf4}
 \pa F_{\delta,\eta} \cap \pa U &\subset & K \cap \pa U\,.
\end{eqnarray}
Moreover,
\begin{equation} \label{osf2}
\lim_{\eta \to 0^+} |F_{\de,\eta} \, \Delta\, E| = 0 \qquad \mbox{for every $\de$}\,,
\end{equation}
and
\begin{equation} \label{osf3}
\begin{split}
\limsup_{\eta \to 0^+} \;& \H^n(\Om \cap \pa F_{\delta,\eta} \setminus \cl(U)) \\& \leq (1+\delta)^n \, \Big( \H^n(\Om \cap \pa^*E) + 2\,\H^n(K \setminus (\cl(U) \cup \pa^*E)) \Big)\,.
\end{split}
\end{equation}
\end{lemma}

Without losing sight of the general picture, we first prove Theorem \ref{proposition Lipschitz minimizing}, and then take care of proving Lemma \ref{l:one_sided_fattening}.

\begin{proof}[Proof of Theorem \ref{proposition Lipschitz minimizing}] Let us recall from \cite[Lemma 3.1]{kms2}, that if $M$ is a smooth hypersurface in $\R^{n+1}$, then there exists a closed set $J \subset M$ with empty interior in $M$ such that a smooth unit normal vector field to $M$ can be defined on $M \setminus J$. Combining this fact with Theorem \ref{thm basic regularity}-(ii), we find that if $(K,E)$ is a generalized minimizer of $\psi(\eps)$, then there exists a subset\footnote{This set $\Sigma$ could be much larger than the singular set of $K$, but is denoted with same letter used for the singular set of $K$ since the notation should be clear from the context.} $\Sigma \subset K$, closed and with empty interior relatively to $K$, such that $K \setminus \Sigma$ is a smooth \emph{orientable} hypersurface in $\R^{n+1}$. We shall denote $\nu$ a smooth unit normal vector field on $K \setminus \Sigma$.

\medskip

Let us fix $\rho>r$ such that $B_\rho(x)\cc \Om\setminus \cl(E)$ and
\begin{equation} \label{choice of rho}
\H^n(K \cap \pa B_\rho(x))=0\,.
\end{equation}
Since $f(K)\cap \pa B_\rho(x)=K\cap \pa B_\rho(x)$, we can apply Corollary \ref{cor:Minkowski final for soap films} to find
\[
\lim_{t\to 0^+}\frac{|U_t(f(K))\cap B_\rho(x)|}{2\,t}=\H^n\big(f(K)\cap B_\rho(x)\big)\,.
\]
By applying the coarea formula to the distance function from $f(K)$, see e.g. \cite[Theorem 18.1, Remark 18.2]{maggiBOOK}, we find that $v(t)=|U_t(f(K)) \cap B_\rho(x)|$ satisfies
\[
v(t)=\int_{0}^t \H^n\big(\pa (U_\eta(f(K))) \cap B_\rho(x)\big) \, d\eta\,,
\]
and is thus absolutely continuous, with
\[
v'(\eta)= \H^n\big(\pa (U_\eta(f(K))) \cap B_\rho(x)\big)\,, \qquad \mbox{for a.e. $\eta > 0$}\,;
\]
moreover, again for a.e. $\eta>0$,
\begin{equation}
  \label{uetaK good}
  \begin{split}
  &\mbox{$U_\eta(f(K))$ is a set of finite perimeter}
  \\
  &\mbox{whose reduced boundary is $\H^n$-equivalent to $\pa[U_\eta(f(K))]$}\,.
  \end{split}
\end{equation}
Therefore, for every $t>0$ there are points of differentiability $\eta_1(t), \eta_2(t) \in \left(0,t\right)$ of $v$ such that \eqref{uetaK good} holds at $\eta=\eta_1(t),\eta_2(t)$, and
\[
v'(\eta_1(t)) \leq \frac{v(t)}{t} \leq v'(\eta_2(t))\,.
\]
Picking any sequence $t_j\to 0^+$, and correspondingly setting $\eta_j=\eta_1(t_j)$, we thus find
\begin{eqnarray} \nonumber
\liminf_{j \to \infty} \H^n\Big(\pa ( U_{\eta_j}(f(K))  )\cap B_\rho(x)\Big)& \leq& \liminf_{j \to \infty} \frac{|U_{t_j}(f(K)) \cap B_\rho(x)|}{t_j}
 \\ \label{final_2}
 &=&2\,\H^n(f(K) \cap B_\rho(x))\,.
\end{eqnarray}
We also notice that since $\{\cl(U_{\eta_j}(f(K)))\}_j$ is a decreasing sequence of sets with monotone limit $\cl(f(K))$, we have
\begin{equation}
  \label{mancava}
  \lim_{j\to\infty}\H^n\Big(\pa B_\rho(x)\cap \cl(U_{\eta_j}(f(K)))\Big)=\H^n\Big(\pa B_\rho(x)\cap \cl(f(K))\Big)=\H^n(\pa B_\rho(x)\cap K)=0\,,
\end{equation}
again thanks to \eqref{choice of rho}. We now pick $\de\in(0,1)$, and define $\{G_j\}_j$ by letting
\begin{equation} \label{almgren_competitors}
G_j =\left( U_{\eta_j}(f(K)) \cap B_\rho (x) \right) \cup F_j \subset \Omega\,,
\end{equation}
with $F_j=F_{\delta,\eta_j}$ as in Lemma \ref{l:one_sided_fattening} with $U = B_\rho(x)$. Since
\begin{equation} \label{boundary containment}
\pa G_j \subset \Big(\pa (U_{\eta_j}(f(K))) \cap B_\rho(x) \Big) \cup \Big( \cl(U_{\eta_j}(f(K))) \cap \pa B_\rho(x) \Big) \cup \pa F_j\,,
\end{equation}
by \eqref{uetaK good}  we see that $\pa G_j$ is $\H^n$-rectifiable for every $j$. Next, we make the following claim
\begin{equation}
  \label{claim}\,\mbox{$\Om \cap \pa G_j$ is $\C$-spanning $W$ for every $j$}\,,
\end{equation}
which implies that $G_j$ is a competitor for the problem $\psi(|G_j|)$. In this way, by
\[
E \, \Delta \, G_j \subset (E\,\Delta\, F_j) \cup U_{\eta_j}(f(K))\,,
\]
and by \eqref{osf2} we find that $|G_j| \to \eps$ as $j \to \infty$, and since $\psi(\eps)$ is lower semicontinuous on $(0,\infty)$, see \cite[Theorem 1.9]{kms}, we conclude that
 \begin{eqnarray*}
 \H^n(\Om\cap\pa^*E) + 2\,\H^n(K\setminus\pa^*E) &=&\psi(\e)
    \le\liminf_{j\to\infty}\psi(|G_j|)
    \\
    &\le&\liminf_{j\to\infty}\H^n(\Om\cap\pa G_j)\,.
  \end{eqnarray*}
In turn, \eqref{boundary containment} implies that
\begin{equation} \label{estimate sum}
\begin{split}
\H^n(\Om \cap \pa G_j) \leq \H^n(& \pa (U_{\eta_j}(f(K)))  \cap B_\rho(x)) + \H^n\Big(\cl(U_{\eta_j}(f(K))) \cap \pa B_\rho(x)\Big)
\\
& + \H^n(\pa F_j \cap \Om \setminus \cl(B_\rho(x))) + \H^n(\pa F_j \cap \pa B_\rho(x))\,,
\end{split}
\end{equation}
and thus, thanks to \eqref{choice of rho}, \eqref{final_2}, \eqref{mancava}, \eqref{osf4} and \eqref{osf3}, we have that
\begin{eqnarray*}
&&\H^n(\Om \cap \pa^* E) + 2\, \H^n (K \setminus \pa ^*E)
\\
&&\leq2\,\H^n(f(K) \cap B_\rho(x))
+ (1+\delta)^n \, \Big\{ \H^n(\Om \cap \pa^*E) + 2\,\H^n(K \setminus (\cl(B_\rho(x)) \cup \pa^*E)) \Big\}\,.
\end{eqnarray*}
By using again \eqref{choice of rho} and letting $\delta \to 0^+$ we deduce
\[
\H^n(K\cap B_\rho(x))\le\H^n(f(K)\cap B_\rho(x))\,.
\]
To complete the proof we are thus left to prove our claim \eqref{claim}.

\medskip

Given $\gamma \in \C$ we want to show that
\begin{equation}
  \label{comoda}
  \g\cap\Om\cap\pa G_j\ne\emptyset\,.
\end{equation}
If $\gamma \cap (K \setminus \cl(B_\rho(x))) \ne \emptyset$, then by \eqref{osf1} we also have $\gamma \cap (\Om\cap \pa F_j \setminus \cl (B_\rho(x))) \ne \emptyset$, and \eqref{comoda} holds. We can then suppose that $\gamma \cap (K \setminus \cl(B_\rho(x)))  = \emptyset$, so that, since $K$ is $\C$-spanning $W$, $\gamma \cap K \cap \cl(B_\rho(x)) \neq \emptyset$.

\medskip

If there is $x_0 \in \gamma \cap K \cap \pa B_\rho(x)$, then necessarily $x_0 \in \pa G_j$. Indeed, $G_j \cap \pa B_\rho(x) = \emptyset$ by construction, so that $x_0 \notin G_j$; on the other hand, since $\{f\ne\id\}\cc B_\rho(x)$, we have that $x_0 \in f(K)$, and thus $x_0 \in \cl(U_{\eta_j}(f(K)) \cap B_\rho(x)) \subset \cl(G_j)$.

\medskip

Hence we can assume $\gamma\cap (K \setminus B_\rho(x)) = \emptyset$, and thus the existence of $x_0 \in \gamma \cap K \cap B_\rho(x)$. By \cite[Lemma 2.2]{kms}, there exists a connected component $\gamma_0$ of $\gamma \cap \cl(B_\rho(x))$ which is diffeomorphic to an interval, whose end-points $p,q$ belong to different connected components of $\pa B_\rho(x) \setminus K$, and such that $\gamma_0 \setminus \{p,q\} \subset B_\rho(x)$. Arguing as in \cite[Proof of Theorem 4, Step 3]{DLGM}, we conclude that in fact $p=f(p)$ and $q=f(q)$ belong to the closures of distinct connected components of $B_\rho(x) \setminus f(K)$, and thus there exists $y_0 \in (\gamma_0\setminus\{p,q\}) \cap f(K)$. Let $u$ be the function $u(y) = \dist(y, f(K)\cap B_\rho(x))$, and consider its restriction to the interval $\gamma_0$. If $\min\{u(p),u(q)\} \leq \eta_j$, then either $p$ or $q$ belongs to $\pa B_\rho(x)\cap\cl(U_{\eta_j}(f(K))\cap B_\rho(x))  \subset \pa G_j$. Otherwise, both $u(p) > \eta_j$ and $u(q) > \eta_j$, whereas $u(y_0)=0$, and thus, by the intermediate value theorem, $\gamma_0 \cap B_\rho(x) \cap \pa(U_{\eta_j}(f(K))) \neq \emptyset$. Thus $\gamma_0\cap\pa G_j\ne\emptyset$, and the proof is complete.
\end{proof}

\begin{proof}[Proof of Lemma \ref{l:one_sided_fattening}] Let us recall that we have set
\begin{eqnarray*}
M&=&K \setminus (\Sigma \cup \pa E \cup \cl(U))=M_0\cup M_1\,,
\\
M_0&=&M\setminus\cl(E)=(K \setminus \Sigma) \setminus (\cl (E) \cup \cl (U))\,,
\\
M_1&=&M\cap E= (K \setminus \Sigma) \cap E\,,
\end{eqnarray*}
and
\[
A_0=g(M_0\times(0,1))\,,\qquad A_1=g(M_1\times(0,1))\,,\qquad F=F_{\de,\eta}=A_0\cup \big(E\setminus\cl(A_1)\big)\,,
\]
where $g:M\times\R\to\R^{n+1}$ and $u:M\to(0,\eta]$ are defined by setting
\begin{eqnarray*}
u(x)&=&\min\Big\{ \eta\,,\,   \frac{\dist(x,\Sigma \cup \pa E \cup \cl(U) \cup W)}{2}\,,\, \frac{\delta}{\|A_M\|(x)}\Big\}\,,
\\
g(x,t)&=&x+t\,u(x)\,\nu(x)\,.
\end{eqnarray*}
We divide the argument in two steps.

\medskip

\noindent {\it Step one:} In this step we prove \eqref{osf4} as well as
  \begin{eqnarray}
    \label{opening -2}
    &&\mbox{$F$ is open with $F\subset\Om \setminus \cl(U)$}\,,
    \\
    \label{opening -1}
    &&(K\setminus \cl (U) )\cup\Big\{x+u(x)\,\nu(x):x\in M\Big\}=\Om\cap\pa F \setminus \cl (U)\,.
  \end{eqnarray}
  Notice that \eqref{opening -1} immediately implies \eqref{osf1}, while \eqref{osf2} follows from $F \, \Delta \, E \subset A_0 \cup\, \cl(A_1)\subset I_\eta(K)$ and the fact that, as $\eta\to 0^+$, $|I_\eta(K)|\to|K|=0$. Therefore in step two we will only have to prove the validity of \eqref{osf3}.

  \medskip

  Since $M_0$ and $M_1$ are relatively open in $M$ and $u$ is positive on $M$, it is easily seen that $A_0$ and $A_1$ are open, and thus that $F$ is open. Since $M\subset\Om\setminus \cl(U)$ and $u(x)<\dist(x,W \cup \cl(U))$ for every $x\in M$, we deduce that $A_0=g(M_0\times(0,1))\subset\Om\setminus\cl(U)$; since trivially $E\subset\Om\setminus\cl(U)$, we have proved \eqref{opening -2}.

  \medskip

  As a preliminary step towards proving \eqref{opening -1}, we show that, for $k=0,1$, we have
  \begin{equation}
    \label{opening 0}
     M_k\cup\big\{x+u(x)\,\nu(x):x\in M_k\big\}\,\,\subset\,\,   \Om \cap \pa A_k \,\,\subset\,\, K\cup\big\{x+u(x)\,\nu(x):x\in M_k\big\}\,.
  \end{equation}
  The first inclusion in \eqref{opening 0} is due to the fact that if $y=x+s\,\nu(x)$ for some $x\in M$ and $|s|< 1/\|A_M\|(x)$, then $x$ and $s$ are uniquely determined in $M$ and $[-1/\|A_M\|(x),1/\|A_M\|(x)]$. The second inclusion in \eqref{opening 0} follows because if $y\in   \Om \cap \pa A_k$, then $y$ is the limit of a sequence $x_j+t_j u(x_j)\,\nu(x_j)$ with $t_j\in(0,1)$, $x_j\in M_k$, $t_j\to t_0\in  \left[0,1\right]$ and $x_j\to x_0\in \cl(M_k)$. If $x_0\in\cl(M_k)\setminus M_k\subset\Sigma\cup\pa E\cup\cl(U)  \cup W$, then $u(x_j) \to 0$ and therefore $y=x_0\in K$. If $x_0\in M_k$ then clearly $t_0\in\{0,1\}$: when $t_0=0$, then $y=x_0\in K$; when, instead, $t_0=1$, then $y=x_0+u(x_0)\,\nu(x_0)$ for $x_0\in M_k$ as claimed.

  \medskip

  We prove the inclusion $\supset$ in \eqref{opening -1} by showing that actually
  \begin{equation}
    \label{fix 4}
    \Om\cap\pa F\,\subset\,\,K\cup\Big\{x+u(x)\,\nu(x):x\in M\Big\}\,.
  \end{equation}
  Since the boundary of the union and of the intersection of two sets is contained in the union of the boundaries, and since the boundary of a set coincides with the boundary of its complement, the inclusion $\pa(\cl(A_1))\subset\pa A_1$ gives
  \begin{eqnarray}\nonumber
    \Om\cap\pa F&\subset&\Om\cap\Big(\pa A_0\cup\pa[E\setminus\cl(A_1)]\Big)\,\subset\,
    \Om\cap\Big(\pa A_0\cup\pa E\cup\pa [\R^{n+1}\setminus\cl(A_1)]\Big)
    \\\label{recalling}
    &=&
    \Om\cap\Big(\pa A_0\cup\pa E\cup\pa (\cl(A_1))\Big)\,\subset\,\Om\cap\big(\pa E\cup\pa A_0\cup\pa A_1\big)\,.
  \end{eqnarray}
  Hence, \eqref{fix 4} follows from $\Om\cap\pa E\subset K$ and \eqref{opening 0}.

  \medskip

  We prove \eqref{osf4}, i.e. $\pa F\cap\pa U\subset K\cap\pa U$. Indeed, $M \cap \cl(U) = \emptyset$ and $u(x) < \dist(x,\cl(U))$ for every $x\in M$ give
  \begin{eqnarray*}
  \pa U\cap \big\{x+u(x)\,\nu(x):x\in M\big\}=\emptyset\,,
  \end{eqnarray*}
  so that \eqref{osf4} follows immediately from \eqref{fix 4}.

  \medskip

  Finally, we complete the proof of \eqref{opening -1} by showing the inclusion $\subset$. Since $\cl(U)\cap\pa E=\emptyset$ and
  $M=K \setminus (\Sigma \cup \pa E \cup \cl(U))$, this amounts to show that
  \begin{eqnarray}
    \label{opening 3}
    M\cup \Big\{x+ u(x)\,\nu(x):x\in M\Big\}&\subset&\Om\cap\pa F \setminus \cl(U)\,,
    \\
    \label{fix 2}
    \Sigma\setminus(\pa E \cup \cl(U))&\subset&\Om\cap\pa F\setminus \cl(U)\,,
    \\
    \label{fix 3}
    \Om\cap\pa E&\subset&\Om\cap\pa F\setminus \cl(U)\,.
  \end{eqnarray}
  {\it Proof of \eqref{opening 3}}: Since $M_0\cap(\cl(E) \cup \cl(U) )=\emptyset$, $M_1\subset E$, and $u(x)<\dist(x,\pa E \cup \cl(U))$ for every $x\in M$,
  we find
  \begin{equation}
    \label{opening 1}
    g\big(M_0\times[0,1]\big)\cap(\cl(E) \cup \cl(U) )=\emptyset\,,\qquad     g\big(M_1\times[0,1]\big)\subset E\,.
  \end{equation}
  Let us notice that, if $X,Y\subset\R^{n+1}$, $V\subset\R^{n+1}$ is open, and $X\cap V=Y\cap V$, then $V\cap\pa X=V\cap\pa Y$. We can apply this remark in the open sets $V=E$ and $V=\R^{n+1}\setminus\cl(E)$, together with $A_0\cap\cl(E)=\emptyset$ and $A_1\subset E$ (both consequences of \eqref{opening 1}), the definition of $F=A_0\cup(E\setminus\cl(A_1))$, and $\cl(U)\cap E=\emptyset$, to first deduce that
  \[
  (\pa F)\setminus\cl(E)=(\pa A_0)\setminus\cl(E)\,,\qquad E\cap\pa F=E\cap\pa A_1\,,
  \]
  and then that
  \begin{equation}
    \label{opening 1.1}
      \Big(\big(\pa A_0\big)\setminus(\cl(E) \cup \cl(U)   )  \Big)\,\cup\,\Big(E\cap\pa A_1\Big)\subset\pa F \setminus \cl(U)\,.
  \end{equation}
  Since \eqref{opening 0} and \eqref{opening 1} give
  \begin{equation}
    \label{opening 2}
    g(M_0\times\{0,1\})\subset \Om \cap  \pa A_0\setminus(\cl(E) \cup \cl (U) ) \,,\qquad g(M_1\times\{0,1\})\subset E\cap \pa A_1\,,
  \end{equation}
  we deduce \eqref{opening 3} from \eqref{opening 1.1} and \eqref{opening 2}.

  \medskip

  \noindent {\it Proof of \eqref{fix 2}}: since $M_1=(K\setminus\Sigma)\cap E$ and $\Sigma$ has empty interior in $K$, we find that $\cl(M_1)\cap E=K\cap E$. At the same time, $M\subset\Om\cap\pa F \setminus \cl(U)$ gives $M_1\cap E\subset E\cap\pa F$ and thus $\cl(M_1)\cap E\subset E\cap\pa F$: hence,
  \[
  \Sigma\cap E\,\,\subset \,\, K\cap E\,\,=\,\,\cl(M_1)\cap E\,\,\subset\,\,\Omega\cap\pa F \setminus \cl(U)\,.
  \]
  Setting for the sake of brevity $T=\Om\setminus(\cl(E)   \cup \cl (U))$, so that $T$ is open, we notice that
  $M_0=(K\setminus\Sigma)\cap T$ implies $\Sigma\cap T\subset \Om \cap \cl(M_0)\cap T=K\cap T$, while $M\subset\Om\cap \pa F \setminus \cl(U)$ and $M_0=M\cap T$ give $\Om \cap \cl(M_0) \cap T\,\subset \,T\cap \pa F $; hence
  \[
  \Sigma\setminus(\cl(E) \cup \cl(U) )\,\,\subset \,\,K\setminus(\cl(E) \cup \cl(U) )\,\,\subset\,\,\Om \cap \pa F\setminus(\cl(E)\cup\cl(U) )\,.
  \]
  By combining the last two displayed inclusions, we obtain \eqref{fix 2}.

 \medskip

 \noindent {\it Proof of \eqref{fix 3}}: since $F$ and $E$ coincide in the complement of $\cl(A_0)\cup\cl(A_1)$, and since $\pa E \cap \cl(U)=\emptyset$, we have
  \[
  \Om\cap\pa E\setminus\big(\cl(A_0)\cup\cl(A_1)\big)\,\,=\,\,\Om\cap\pa F\setminus\big(\cl(A_0)\cup\cl(A_1)\cup\cl(U)\big)\,\,\subset\,\,\Om\cap\pa F\setminus\cl(U)\,.
  \]
  Now let $y\in\Om\cap\pa E\cap\cl(A_1)$: since $A_1=g(M_1\times(0,1))$ and $y\not\in g(M_1\times[0,1])$ by \eqref{opening 1}, we find that $y$ is in the closure of $M_1$, and thus of $M$, relatively to $K$: thus $y\in\Om\cap\cl(M)\setminus \cl(U)$; at the same time, by \eqref{opening 3}, we have $M\subset\Om\cap\pa F \setminus \cl(U)$ and thus $\Om\cap\cl(M)\setminus \cl(U)\subset\Om\cap\pa F \setminus \cl(U)$; combining the two facts,
  \[
  \Om\cap\pa E\cap\cl(A_1)\subset \Om\cap\pa F \setminus \cl(U)\,.
  \]
  We argue similarly to show that $\Om\cap\pa E\cap\cl(A_0)\subset\Om\cap\pa F \setminus \cl(U)$ and thus prove \eqref{fix 3}.

  \medskip

  \noindent {\it Step two}. We prove \eqref{osf3}. First, we notice that thanks to \eqref{opening -1}
  \begin{equation} \label{energy estimate opening1}
  \H^n(\Om\cap\pa F \setminus \cl(U))\le\H^n(K \setminus \cl(U))+\H^n\Big(\big\{x+u(x)\,\nu(x):x\in M\big\}\Big)\,.
  \end{equation}
  Since $\dist(x,\Sigma\cup\pa E \cup \cl(U) \cup W)>0$ and $\|A_M\|(x)<\infty$ for every $x\in M$, we find that
  \[
  M_\eta=\big\{x\in M:u(x)=\eta\big\}=\Big\{x\in M\,\colon\,\|A_M\|(x)\le\frac\de\eta\,,\,\dist(x,\Sigma\cup\pa E\cup \cl(U)\cup W)\ge2\eta\Big\}
  \]
  is monotonically increasing towards $M$ as $\eta\to 0^+$. Moreover, $x\mapsto x+u(x)\,\nu(x)=x+\eta\,\nu(x)$ is smooth on $M_\eta$, and if $\k_i$ are the principal curvatures of $M$ with respect to $\nu$,
  \begin{equation} \label{energy estimate opening2}
  \H^n\Big(\big\{x+u(x)\,\nu(x):x\in M_\eta\big\}\Big)=\int_{M_\eta}\prod_{i=1}^n(1+\eta\,\k_i)
  \le(1+\de)^n\,\H^n(M_\eta)\le(1+\de)^n\,\H^n(M)\,.
  \end{equation}
  Letting $\eta\to 0^+$, $g(M_\eta\times\{1\})=\{x+u(x)\,\nu(x):x\in M_\eta\}$ is increasingly converging to $g(M\times\{1\}) = \{x+u(x)\,\nu(x):x\in M\}$, so that \eqref{energy estimate opening2} yields
\begin{equation} \label{estimate M up}
\H^n(\{ x+u(x)\,\nu(x)\,\colon\, x\in M  \}) \leq (1+\de)^n\,\H^n(M)\,,
\end{equation}
  and therefore from \eqref{energy estimate opening1} we deduce
    \begin{equation} \label{energy estimate opening final}
  \limsup_{\eta\to 0^+}\H^n(\Om\cap\pa F \setminus \cl(U))\le\H^n(K \setminus \cl(U))+(1+\de)^n\,\H^n(M)\,.
  \end{equation}
  Finally, \eqref{osf3} follows from \eqref{energy estimate opening final} once we observe that $M=K\setminus(\Sigma\cup\pa E \cup \cl(U))\subset K\setminus(\pa^*E \cup \cl(U) )$, so that
  \begin{eqnarray*}
  \H^n(K\setminus \cl(U))+\H^n(M)&=&\H^n(\Om\cap\pa^*E)+\H^n(K\setminus(\pa^*E \cup \cl(U)))+\H^n(M)
  \\
  &\le&\H^n(\Om\cap\pa^*E)+2\,\H^n(K\setminus(\pa^*E \cup \cl(U) ))\,,
  \end{eqnarray*}
  as required.
\end{proof}

\section{Wetting competitors and exclusion of points of type $Y$} \label{s:noY} By Theorem \ref{proposition Lipschitz minimizing}, $K$ is an Almgren minimal set in $\Om\setminus\cl(E)$. As we shall see in the next section, this property is compatible with $K$ containing $(n-1)$-dimensional submanifolds of $Y$-points, that is, points such that $K$ is locally diffeomorphic to a cone of type $Y$ in $\R^{n+1}$. The goal of this section is to show that, for reasons related to the specific properties of the variational problem $\psi(\e)$, such points cannot exist. As explained in detail in the next section, this bit of information will prove crucial in closing the proof of Theorem \ref{t:main}.

\begin{theorem}\label{prop:noYx}
    If $(K,E)$ is a generalized minimizer of $\psi(\e)$, then there cannot be $x_0\in K\setminus\cl(E)$ such that there exist $\a\in(0,1)$, a $C^{1,\a}$-diffeomorphism $\Phi:\R^{n+1}\to\R^{n+1}$ with $\Phi(0)=x_0$, and $r_0>0$ such that, setting $A=\Phi(B_{r_0})$,
    \begin{equation} \label{K is a perturbation of Y}
    \Phi\Big((\bY^1\times\R^{n-1})\cap B_{r_0}\Big)=A\cap K\,,
    \end{equation}
    and, in addition,
    \begin{equation} \label{e:properties of diffeo}
    \nabla\Phi(0)=\Id\,,\qquad
    \| \nabla \Phi - \Id \|_{C^0(B_r)} \leq C\, r^\a\,,\qquad\forall r<r_0\,.
    \end{equation}
\end{theorem}

\begin{proof} The proof is achieved by showing that, up to decrease the value of $r_0$, there exist a constant $c_{\bY}=c_{\bY}(n)>0$ and a set $G \subset \Omega$ with
\begin{equation} \label{contra:the set G}
G \in \E\,, \quad |G|=\eps\,, \quad \mbox{$\Om \cap \pa G$ is $\C$-spanning $W$}\,,
\end{equation}
such that
\begin{equation} \label{contra:the estimate}
\H^n(\Om \cap \pa G) \leq \F(K,E) - c_{\bY}\,r_0^n\,.
\end{equation}
The set $G$, defined in \eqref{the competitor!} below, is constructed in three stages, that we introduce as follows. We pick $x^*\in\Om\cap\pa^*E$, so that, by Theorem \ref{thm basic regularity}, we can find an open cylinder $Q^*$ of height and radius $r^*>0$, centered at $x^*$, and with axis along $\nu_E(x^*)$, such that $E\cap Q^*$ is the subgraph of a smooth function $u$ defined over the cross section $D^*$ of $Q^*$, and such that the graph of $u$ has mean curvature $\l$ in the orientation induced by $\nu_E(x^*)$ (here $\l$ is the Lagrange multiplier of $(K,E)$).
\begin{figure}
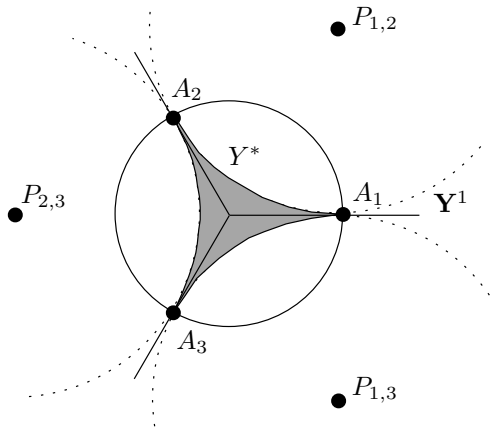
\caption{\small{The construction in step one of the proof of Theorem \ref{prop:noYx}.}}
\label{fig ystar}
\end{figure}
Up to decrease $r_0$ and $r_*$, we can make sure that $A=\Phi(B_{r_0})$ and $Q^*$ lie at positive distance, so that modification of $(K,E)$ compactly supported in these two regions will not interact. We then argue in three stages: in the first stage (first three steps of the proof), we modify $(K,E)$ by replacing the collapsed surface $K\cap A$ with an open set of the form $\Phi(\Delta_{r_0})$, where $\Delta_{r_0}\subset B_{r_0}$ is constructed so to achieve an ${\rm O}(r_0^n)$-gain in area, at the cost of an ${\rm O}(r_0^{n+1})$-increase in volume -- this is possible, of course, only because we are assuming that $x_0$ is a $Y$-point; in the second stage (step four of the proof), we construct a one-parameter family of modifications $\{E_t\}_{t\in(0,t_0)}$ of $E$, all supported in $Q^*$, with $|E_t|=|E|-t$, and such that the area increase from $\pa E$ to $\pa E_t$ is of order ${\rm O}(t)$; in the final stage, we apply Lemma \ref{l:one_sided_fattening} to the element of $\KK$ constructed in stage one, and then use the volume-fixing variation of stage two to create the competitor $G$ which will eventually give the desired contradiction.

\medskip

\noindent {\it Step one:} There exist positive constants $v_0$ and $c_0$, such that for every $\delta >0$, there is an open subset $Y^*_\delta \subset B_\delta \subset \R^2$ such that
\begin{eqnarray} \label{2D:geometry}
\bY^1 \cap B_\delta \subset Y^*_\delta\,, &\qquad &  \cl(Y^*_\delta) \cap \pa B_\delta = \bY^1 \cap \pa B_\delta\,,
\\
 \label{e:2D gain}
 \quad |Y^*_\delta| = v_0 \,\delta^2\,, & \qquad & \H^1(\pa Y^*_\delta ) = 2\, \H^1(\bY^1 \cap B_\delta) - c_0\,\delta\,.
\end{eqnarray}
This results from an explicit construction, see Figure \ref{fig ystar}. By scale invariance of the statement, we can assume that $\delta=1$.  Let $\{A_i\}_{i=1}^3 = \bY^1 \cap \pa B_1$, and let $\{P_{1,2}, P_{2,3}, P_{1,3}\}$ be defined as follows: for $i,j \in \{1,2,3\}$, $i < j$, $P_{i,j}$ is the intersection of the straight lines $\ell_i$ and $\ell_j$ tangent to $\pa B_1$ and passing through $A_i$ and $A_j$, respectively. We also let $S_{i,j}$ be the closed disc sector centered at $P_{i,j}$ and corresponding to the arc $A_iA_j$. Finally, we define
\begin{equation} \label{Y-competitor}
Y^*=Y^*_1= B_1 \setminus \bigcup_{i<j} S_{i,j}\,.
\end{equation}
It is easily shown that \eqref{2D:geometry} holds with
\begin{equation} \label{exact values}
v_0=\frac{3}{2}(2\,\sqrt{3}-\pi)\,, \qquad c_0 = 6-\pi\,\sqrt{3}\,.
\end{equation}

\medskip

\noindent {\it Step two:} We adapt to higher dimensions the construction of step one, see the set $\Delta_{r_0}$ defined in \eqref{deltar0} below. We assign coordinates $x = (z,y) \in \R^2\times\R^{n-1}$ to points $x \in \R^{n+1}$, so that $(0,y)$ is the component of the vector $x$ along the spine of the cone $\bY^1 \times \R^{n-1} \subset \R^2 \times \R^{n-1}$, and $|z|$ is the distance of $x$ from the spine. We observe that, if $\mathbf{p} \colon \R^{n+1} \to \R^{n-1}$ denotes the orthogonal projection operator onto the spine, then the \emph{slice} of $B_{r_0}$ with respect to $\mathbf{p}$ at $y \in \R^{n-1}$ is given by
\begin{equation} \label{fading disks}
B_{r_0} \cap \mathbf{p}^{-1}(y) =
\begin{cases}
     (0,y) + B^2_{\sqrt{r_0^2-|y|^2}}          & \mbox{if $|y| < r_0$} \\
 \emptyset              & \mbox{otherwise}\,,
\end{cases}
\end{equation}
where $B^2_{\rho}$ is the disc of radius $\rho$ in $\R^2 \times \{0\}$. Analogously, the slice of $(\bY^1 \times \R^{n-1}) \cap B_{r_0}$ with respect to $\p$ at $y$ is
\begin{equation} \label{fading Ys}
(\bY^1 \times \R^{n-1}) \cap B_{r_0} \cap \mathbf{p}^{-1}(y) =
\begin{cases}
(0,y) + \bY^1 \cap B^2_{\sqrt{r_0^2-|y|^2}} & \mbox{if $|y| < r_0$}\\
\emptyset &\mbox{otherwise}\,.
\end{cases}
\end{equation}
For $\tau \in \left(0,1/2\right)$ to be chosen later, we pick $g \in C^\infty_{c}(\left[0,\infty\right))$ such that
\begin{eqnarray*}
g \equiv \tau \;\; \mbox{in $\left[ 0,\tau \right]$}\,, & \qquad & g \equiv 0 \;\; \mbox{in $\left[1,\infty\right)$}\,,\\
0<g(t) \leq \sqrt{1-t^2} \;\; \mbox{in $\left[0,1\right)$}\,, & \qquad & g' \leq 0 \mbox{ and }|g'|\leq 2\,\tau \;\;\mbox{everywhere}\,,
\end{eqnarray*}
and set $g_{r_0}(s) = r_0\,g(s/r_0)$, see
\begin{figure}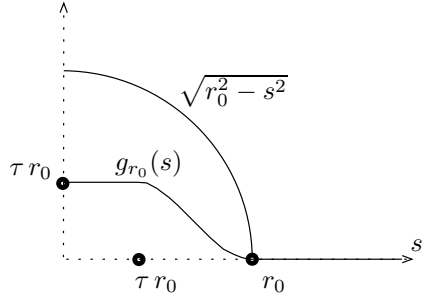\caption{\small{The dampening function $g_{r_0}$. The set $U_{r_0}$ is obtained by rotating the graph of $g_{r_0}$ around the spine of $\bY^1\times\R^{n-1}$.}}\label{fig gr0}\end{figure}
Figure \ref{fig gr0}. Next, let $U_{r_0}$ define the open tubular neighborhood
\begin{equation}
U_{r_0} = \Big\{  x=(z,y) \in \R^{n+1} \,\colon\, \mbox{$|y|<r_0$ and $|z| < g_{r_0}(|y|)$}   \Big\}\,,
\end{equation}
and notice that $U_{r_0} \subset B_{r_0}$ by \eqref{fading disks} and the properties of $g$. Finally, we define the set
\begin{equation}\label{deltar0}
\Delta_{r_0} = \Big\{ x=(z,y) \in \R^{n+1} \, \colon \, \mbox{$|y|<r_0$ and $z \in Y^*_{g_{r_0}(|y|)}$} \Big\}\,.
\end{equation}
We claim that $\Delta_{r_0}$ is an open subset of $U_{r_0}$ (thus of $B_{r_0}$), with $\H^n$-rectifiable boundary, and such that
\begin{eqnarray}
  \label{first containment}
(\bY^1 \times \R^{n-1}) \cap U_{r_0}&\subset&\Delta_{r_0}\,,
\\
\label{Delta volume}
|\Delta_{r_0}| &\le& C(n)\,v_0\,r_0^{n+1}\,,
\\
\label{Delta perimeter deficit}
\H^n(\pa\Delta_{r_0})&\le& 2\,\H^n\Big((\bY^1\times\R^{n-1}) \cap U_{r_0}\Big)-C(n)\,c_0\,r_0^n\,.
\end{eqnarray}
Only \eqref{Delta volume} and \eqref{Delta perimeter deficit} require a detailed proof. To compute the volume of $\Delta_{r_0}$, we apply Fubini's theorem, step one, and the definition of $g_{r_0}$, to get
\[
|\Delta_{r_0}| =  \int_{B_{r_0}^{n-1}} |Y^*_{g_{r_0}(|y|)}| \, dy = v_0\, \int_{B_{r_0}^{n-1}} g_{r_0}(|y|)^2\,dy = v_0 \, r_0^{n+1}\, \int_{B_1^{n-1}} g(|w|)^2 \, dw\,.
\]
Similarly, we use the coarea formula (see e.g. \cite[Theorem 2.93 and Remark 2.94]{AFP}) to write
\begin{equation} \label{coarea formula}
\begin{split}
\H^n(\pa \Delta_{r_0}) &- 2\, \H^n((\bY^1\times\R^{n-1}) \cap U_{r_0})
\\
&= \int_{B_{r_0}^{n-1}}\,dy \int_{\pa Y^*_{g_{r_0}(|y|)}} \frac{d\H^1(z)}{\mathrm{C}_{n-1}(\nabla^{\pa\Delta_{r_0}}\mathbf{p}(z))}  - 2\,\int_{B_{r_0}^{n-1}} \H^1(\bY^1 \cap B^2_{g_{r_0}(|y|)}) \, dy
\end{split}
\end{equation}
where $\nabla^{\pa\Delta_{r_0}}\mathbf{p}$ is the tangential gradient of $\mathbf{p}$ along $\pa \Delta_{r_0}$, and where $\mathrm{C}_{n-1}(L)$ is the $(n-1)$-dimensional coarea factor of a linear map $L \colon \R^n \to \R^{n-1}$. Standard calculations show that, for every $y \in B_{r_0}^{n-1}$,
\[
\mathrm{C}_{n-1}(\nabla^{\pa\Delta_{r_0}} \mathbf{p}(z)) = \left(  1 + |g_{r_0}'(|y|)|^2 \right)^{-\frac{n-1}{2}}\qquad \mbox{for $\H^1$-a.e. $z \in \pa Y^*_{g_{r_0}(|y|)}$}\,,
\]
so that \eqref{coarea formula} allows to estimate
\begin{eqnarray*}
\H^n(\pa \Delta_{r_0}) & - & 2\, \H^n((\bY^1\times\R^{n-1}) \cap U_{r_0})  \\
&\leq & \int_{B_{r_0}^{n-1}} \left(  ( 1 + 4\,\tau^2  )^{\frac{n-1}{2}} \, \H^1(\pa Y^*_{g_{r_0}(|y|)}) - 2\,\H^1(\bY^1 \cap B^2_{g_{r_0}(|y|)}) \right) \,dy\\
&=& \left(  (1+4\,\tau^2)^{\frac{n-1}{2}} \,\H^1(\pa Y^*_1) - 2\, \H^1(\bY^1 \cap B_1^2)  \right) \, r_0^n \, \int_{B_1^{n-1}} g(|w|) \, dw \,,
\end{eqnarray*}
and thus by \eqref{e:2D gain} and provided $\tau$ is sufficiently small depending on $n$, $c_0$ and $\H^1(\bY^1\cap B_1^2)$,
\[
\H^n(\pa \Delta_{r_0})  -  2\, \H^n((\bY^1\times\R^{n-1}) \cap U_{r_0}) \leq  -  \frac{c_0}{2} \,r_0^n\, \int_{B_1^{n-1}} g(|w|) \, dw\,,
\]
which gives \eqref{Delta perimeter deficit}.

\medskip

\noindent {\it Step three:} By step two and the properties of $\Phi$, we have that $\Phi(\Delta_{r_0})$ is an open subset of $\Phi(U_{r_0}) \subset \Phi(B_{r_0})=A$ with $\H^n$-rectifiable boundary $\pa\Phi(\Delta_{r_0})=\Phi(\pa\Delta_{r_0})$, and such that
\begin{equation} \label{refined containment}
K \cap \Phi(U_{r_0})=\Phi\Big((\bY^1\times\R^{n-1})\cap U_{r_0}\Big) \subset \Phi(\Delta_{r_0})
\end{equation}
thanks to \eqref{first containment}.
Moreover, up to possibly taking a smaller value for $r_0$, \eqref{Delta volume} and \eqref{Delta perimeter deficit}, together with the area formula and \eqref{e:properties of diffeo}, guarantee the existence of constants $v_{\bY} = v_{\bY}(n) > 0$ and $c_{\bY} = c_{\bY}(n) >0$ such that
\begin{align} \label{Phi Delta volume}
|\Phi(\Delta_{r_0})| \;\; &\le \;\; v_{\bY}\,r_0^{n+1}\,, \\ \label{Phi Delta perimeter deficit}
\H^n(\pa (\Phi(\Delta_{r_0}))) - 2\, \H^n(K \cap \Phi(U_{r_0})) \;\; &\leq \;\;  - 3\,c_{\bY} \, r_0^n\,.
\end{align}

\medskip

\noindent {\it Step four:} We recall the following construction from \cite[Proof of Theorem 2.8]{kms2}; see
\begin{figure}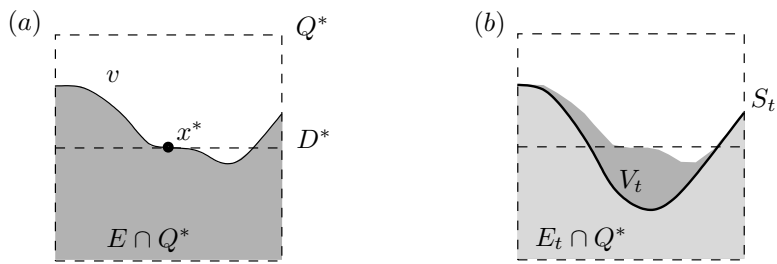\caption{\small{The volume-fixing variations constructed in step four. The surface $S_t$ has been depicted with a bold line.}}\label{fig Vt}\end{figure}
Figure \ref{fig Vt}. Fix a point $x^*\in\Omega \cap \pa^*E$, and let $\nu^*=\nu_E(x^*)$. Theorem \ref{thm basic regularity} guarantees then the existence of a radius $r^* > 0$ such that, denoting by $Q^*$ the cylinder of center $x^*$, axis $\nu^*$, height and radius $r^*$, and by $D^*$ its $n$-dimensional cross-section passing through $x^*$, we have $\cl(Q^*) \cap \cl(\Phi(B_{r_0}))=\emptyset$ and
 \begin{eqnarray} \label{cilindro filling}
    E \cap \cl(Q^*)&=&\Big\{z+h\,\nu^*\,\colon\, z\in \cl(D^*)\,, \; -r^*\leq h<v(z)\Big\}\,, \label{cilindro set}
    \\
    K\cap\cl(Q^*)=\pa E\cap\cl(Q^*)&=&\Big\{z+v(z)\,\nu^*\,\colon\, z\in \cl(D^*)\Big\}\,,
  \end{eqnarray}
  for a smooth function $v \colon \cl(D^*) \to \R$ solving, for $\lambda \le 0$ (the non-positivity of $\lambda$ is not important here, but it holds thanks to the main result in \cite{kms2}),
  \begin{equation}
    \label{pa E is cmc}
      -\Div\bigg(\frac{\nabla v}{\sqrt{1+|\nabla v|^2}}\bigg)=\l\quad\mbox{on $D^*$}\,,\qquad
  \max_{\cl(D^*)}|v|\le \frac{r^*}2\,.
  \end{equation}
  We choose a smooth function $w \colon\cl(D^*)\to\R$ with
  \begin{equation} \label{volume fixing variation}
  w=0\quad\mbox{on $\pa D^*$}\,,\qquad
  w>0\quad\mbox{on $D^*$}\,,\qquad\int_{D^*}w=1\,,
  \end{equation}
  and then define, for $t>0$, an open set $V_t$ by setting
  \begin{equation} \label{cave}
  V_t=\Big\{ z+h\,\nu^*\,\colon\, z\in D^*\,,v(z)-t\,w(z) < h < v(z)\Big\}\,.
  \end{equation}
  For $t$ small enough (depending only on $r^*$ and on the choice of $w$) we have that $V_t\subset E \cap Q^*$, with
  \begin{equation}
  \pa V_t\cap\pa Q^*=K\cap \pa Q^*=\left\lbrace z+v(z)\,\nu^*\,\colon\, z\in \pa D^*\right\rbrace\,.
  \end{equation}
Furthermore, if we let $S_t$ denote the closed set
\begin{equation} \label{d:S_surface}
S_t = \Big\{ z + (v(z) - t w(z))\, \nu^* \, \colon \, z \in \cl(D^*) \Big\}\,,
\end{equation}
  it is easily seen that for $t<t_0$
  \begin{equation}
    \label{volume fixing volume and perimeter}
      |V_t|=t\,,\qquad \H^n(S_t) = \H^n (S_t \cap \cl(Q^*)) = \H^n(\pa E \cap \cl(Q^*)) - \lambda\, t + {\rm O}(t^2)\,,
  \end{equation}
  where we have used $\int_{D^*} w=1$, $w=0$ on $\pa D^*$, and \eqref{pa E is cmc}. In particular, setting
  \begin{equation} \label{E deflated}
   E_t = E \setminus \cl(V_t)\,,
  \end{equation}
we have (see \cite[Equation (3.37)]{kms2})
\begin{equation} \label{properties of Et}
|E_t| = |E| -t\,, \qquad   \pa E_t \cap \cl(Q^*) = S_t\,,
\end{equation}
so that \eqref{volume fixing volume and perimeter} reads
\begin{equation} \label{perimeter estimate bulk}
\H^n ( \pa E_t \cap\cl( Q^*) ) = \H^n(\pa E \cap \cl(Q^*)) + |\lambda|\,t + {\rm O}(t^2)\,.
\end{equation}
Finally, if needed, we further reduce the value of $r_0$ (this time also depending on $|\l|$) in order to entail
\begin{equation} \label{fixing radius}
\H^n(\pa E_t \cap \cl(Q^*)) \leq \H^n(\pa E \cap \cl(Q^*)) + c_{\bY}\,r_0^n \qquad\forall\,t \leq 2\,v_{\bY}\,r_0^{n+1}\,.
\end{equation}

\medskip

\noindent {\it Step five:} Without loss of generality, we can assume that $r^*<\dist(x^*, \cl(K) \setminus \pa^*E)/2$. In this way, provided $\eta$ is small enough in terms of $r^*$, we can enforce that
\begin{equation}
  \label{oddio}\mbox{$I_\eta(\cl(K)\setminus \pa^*E)$ and $\cl(Q^*)$ lie at positive distance.}
\end{equation}
We thus apply the construction of Lemma \ref{l:one_sided_fattening} to $(K,E)$ with the open set $U = \Phi(U_{r_0})$, and correspondingly define the function $u$ and the sets $M_0,M_1,M,A_0,A_1$. After choosing $\delta$ and $\eta$ sufficiently small in terms of $r_0$, we can achieve
\begin{align} \label{condition1}
&(\cl (A_0) \cup \cl(A_1)    ) \cap \cl (Q^*) = \emptyset\,, \qquad   |A_0| + |A_1| \leq \frac{v_{\bY}\,r_0^{n+1}}{4}\\ \label{condition2}
&\H^n\left( \{  x + u(x)\,\nu(x) \, \colon \, x \in M     \}  \right) \leq \H^n(M) + c_{\bY}\,r_0^n\,.
\end{align}
Since $\cl(A_0)\cup\cl(A_1)\subset I_\eta(K \setminus \pa E)$, the first condition in \eqref{condition1} is immediate from \eqref{oddio}, while the second condition follows from $|I_\eta(K)|\to|\cl(K)|=0$ as $\eta\to 0^+$. Finally, \eqref{condition2} is satisfied for $\delta$ sufficiently small (in terms of $r_0$, $n$ and $\H^n(K)$) thanks to \eqref{estimate M up}.

\medskip

\noindent {\it Step six:} We apply step four  with
\begin{equation} \label{volume deficit}
t = |\Phi(\Delta_{r_0})| + |A_0| - |A_1| \in \left( 0, 2\,v_{\bY}\, r_0^{n+1} \right]\,.
\end{equation}
In particular, \eqref{fixing radius} holds for the corresponding set $E_t$, and we can finally define the competitor
\begin{equation} \label{the competitor!}
G =  \Phi(\Delta_{r_0}) \cup F\,, \qquad \mbox{where $F=A_0 \cup \left( E_t \setminus \cl(A_1) \right)$}\,.
\end{equation}
We verify now that $G$ satisfies the properties \eqref{contra:the set G} and \eqref{contra:the estimate}. First, we observe that $\Phi (\Delta_{r_0}) \subset \Phi (U_{r_0})$, whereas, by Lemma \ref{l:one_sided_fattening} and given that $E_t \subset E$, one has $F \subset \Om \setminus \cl(\Phi(U_{r_0}))$, so that $\Phi(\Delta_{r_0})$ and $F$ are two \emph{disjoint} open subsets of $\Omega$. In particular, $G \subset \Om$ is open and, as a consequence of \eqref{properties of Et} and \eqref{volume deficit},
\begin{equation} \label{volume of competitor}
\begin{split}
|G| = |\Phi(\Delta_{r_0})| + |F| \;&=\; |\Phi(\Delta_{r_0})| + |A_0| + |E_t| - |\cl (A_1)|\\
&=\; |\Phi(\Delta_{r_0})| + |A_0| - |\cl(A_1)| + |E| - t \\
&=\, |E|\,.
\end{split}
\end{equation}
Since $\pa G\subset\pa[\Phi(\Delta_{r_0})]\cup\pa F$, recalling the last inclusion in \eqref{recalling} (which in the present case holds with $E_t$ in place of $E$) and noticing that $\cl(\Phi(\Delta_{r_0})) \subset \Omega$, we obtain
\begin{equation} \label{first inclusion}
\Om\cap\pa G\,\, \subset\,\,\pa \Phi(\Delta_{r_0})\cup\Big\{\Om\cap\Big(\pa A_0  \cup \pa A_1 \cup \pa E_t\Big)\Big\}\,.
\end{equation}
and, in particular, $\pa G$ is $\H^n$-rectifiable. Moreover, for $k=0,1$, by \eqref{opening 0} and by $I_\eta(K  \setminus \pa E)\cap \cl(Q^*)=\emptyset$, we get
\[
\Om \cap \pa A_k  \subset  K \setminus (  \Phi(U_{r_0}) \cup   Q^*   ) \cup \Big\{ x + u(x)\,\nu(x)\,\colon\,x \in M_k \Big\}
\]
while
\begin{eqnarray*}
\Om \cap \pa E_t & \subset & [\Om \cap \pa E_t \cap \cl (Q^*)] \cup [(\Om \cap \pa E) \setminus \cl(Q^*)]
 \\
 &\subset&S_t \cup [K \setminus (\Phi (U_{r_0}) \cup Q^*)]
\end{eqnarray*}
so that the $\subset$-inclusion in the following identity
\begin{equation} \label{boundary identity}
\Om \cap \pa G = \pa \Phi(\Delta_{r_0}) \cup (K \setminus (  \Phi(U_{r_0})  \cup Q^*  )) \cup  S_t \cup \Big\{ x + u(x)\,\nu(x)\,\colon\, x \in M \Big\}
\end{equation}
follows from \eqref{first inclusion}. To complete the proof of \eqref{boundary identity} we will show that
\begin{eqnarray} \label{2inclusion:Delta}
\pa \Phi(\Delta_{r_0}) & \subset & \Om \cap \pa G\,,\\ \label{2inclusion:bdry ball}
K \cap \pa (\Phi(U_{r_0})) &\subset & \Om \cap \pa G\,, \\ \label{2inclusion:M}
M \cup \big\{  x + u(x)\,\nu(x) \, \colon \, x \in M \big\}&\subset & \Om \cap \pa G\,, \\ \label{2inclusion:singular}
\Sigma \setminus ( \pa E \cup \cl (\Phi (U_{r_0})) ) &\subset & \Om \cap \pa G\,, \\\label{2inclusion:bdry set out cyl}
(\Om \cap \pa E) \setminus \cl(Q^*) &\subset & \Om \cap \pa G\,, \\  \label{2inclusion:S}
S_t &\subset & \Om \cap \pa G\,.
\end{eqnarray}
\textit{Proof of \eqref{2inclusion:Delta}:} it readily follows from the fact that $\cl(\Phi(\Delta_{r_0})) \subset \Om \cap \cl (G)$ together with $F\cap \cl (\Phi(\Delta_{r_0}))=\emptyset$. \textit{Proof of \eqref{2inclusion:bdry ball}:} since $K \cap \pa (\Phi(U_{r_0})) \subset \Om \setminus G$, we only have to prove that $K \cap \pa (\Phi(U_{r_0})) \subset \cl (G)$. Since $K \cap \cl (\Phi(U_{r_0})) = \Phi(\bY^1\times\R^{n-1}) \cap \cl (\Phi(U_{r_0}))$, any $x \in K \cap \pa (\Phi(U_{r_0}))$ is a limit of points $\Phi(z_h)$ with $z_h \in (\bY^1 \times \R^{n-1}) \cap U_{r_0} \subset \Delta_{r_0}$ by \eqref{first containment}. In particular, $x$ is a limit of points in $\Phi(\Delta_{r_0}) \subset G$. \textit{Proof of \eqref{2inclusion:M}, \eqref{2inclusion:singular}, and \eqref{2inclusion:bdry set out cyl}:} since $E_t \setminus \cl(Q^*)=E\setminus \cl(Q^*)$, the sets appearing on the left-hand sides of \eqref{2inclusion:M}, \eqref{2inclusion:singular}, and \eqref{2inclusion:bdry set out cyl} are all subsets of $\Om \cap \pa F \setminus \cl(\Phi(U_{r_0}))$ as a consequence of \eqref{opening 3}, \eqref{fix 2}, and \eqref{fix 3}, respectively. \textit{Proof of \eqref{2inclusion:S}:} By construction $G\cap Q^*=E_t\cap Q^*$ so that $Q^*\cap\pa G=Q^*\cap\pa E_t$; since $S_t\subset\cl(Q^*)\cap\pa E_t$ we conclude the proof of \eqref{2inclusion:S}, and thus of \eqref{boundary identity}.

\medskip

\noindent {\it Conclusion:} We first prove \eqref{contra:the estimate}. Without loss of generality assume that $r^*$ is such that $\H^n(\pa^*E\cup\pa Q^*)=0$. By \eqref{boundary identity}, \eqref{Phi Delta perimeter deficit}, \eqref{fixing radius}, \eqref{condition2}, and the fact that $M\subset K\setminus(\pa E\cup\cl(\Phi(U_{r_0})))$, we find
\begin{eqnarray*}
\H^n(\Om \cap \pa G) & \leq & \H^n(\pa \Phi(\Delta_{r_0}))
\\
&&+ \H^n((\Om \cap \pa^*E) \setminus Q^*) + \H^n ((K \setminus \pa^*E)\setminus \Phi (U_{r_0}))
\\
&&+ \H^n (S_t) + \H^n(\{x + u(x) \, \nu(x) \, \colon \, x \in M\})
\\
&\le& 2\,\H^n(K\cap\Phi(U_{r_0}))-3\,c_{\bY}\,r_0^n
\\
&&+\H^n((\Om \cap \pa^*E) \setminus Q^*)+ \H^n ((K \setminus \pa^*E)\setminus \Phi (U_{r_0}))
\\
&&+\H^n(\pa^*E\cap\cl(Q^*))+c_{\bY}\,r_0^n+\H^n(M)+c_{\bY}\,r_0^n
\\
&\leq & 2\, \H^n(K \setminus \pa^*E) + \H^n (\Om \cap \pa^*E) - c_{\bY}\,r_0^n\,,
\end{eqnarray*}
that is \eqref{contra:the estimate}. To complete the argument we finally prove that $\Om \cap \pa G$ is $\C$-spanning $W$. To this aim, pick $\gamma \in \C$. If $\gamma \cap K \setminus ( \Phi(U_{r_0})  \cup Q^*  ) \neq \emptyset$, then also $\gamma \cap \pa G \ne \emptyset$ by \eqref{boundary identity}. If $\gamma \cap K \cap Q^* \ne \emptyset$, then also $\gamma \cap \pa E \cap Q^* \ne \emptyset$, and thus also $\gamma \cap S_t \ne \emptyset$ as a consequence of \cite[Lemma 2.3]{kms} since $S_t$ is a diffeomorphic image of $\pa E \cap \cl(Q^*)$: hence, $\gamma \cap \pa G \ne \emptyset$, again by \eqref{boundary identity}. We can therefore assume that $\gamma \cap K \setminus \Phi(U_{r_0}) = \emptyset$, and thus, since $K$ is $\C$-spanning $W$, that there exists $x \in \gamma \cap K \cap \Phi(U_{r_0}) \subset \gamma \cap \Phi(\Delta_{r_0})$, where in the last inclusion we have exploited \eqref{refined containment}. Since $\Phi(\Delta_{r_0})$ is contractible and, as consequence of $\ell<\infty$, $\gamma$ is homotopically non-trivial in $\Omega$, $\gamma$ must necessarily intersect $\R^{n+1} \setminus \Phi(\Delta_{r_0})$, and thus, by continuity, $\gamma \cap \pa \Phi(\Delta_{r_0})\ne\emptyset$. Since $\pa \Phi(\Delta_{r_0})\subset\Om\cap\pa G$, we have completed the proof.
\end{proof}

\section{Regularity theory and conclusion of the proof of Theorem \ref{t:main}}\label{s:graph}

\subsection{Blow-ups of stationary varifolds}\label{section varifolds} We say that $V_0$ is an {\bf integral $n$-cone} in $\R^{n+1}$ if $V_0=\var(\bC,\theta_0)$ for a closed locally $\H^n$-rectifiable cone $\bC$ in $\R^{n+1}$ (so that $\l\,x\in\bC$ for every $x\in\bC$ and $\l>0$), and  a zero-homogenous multiplicity function $\theta_0$ (so that $\theta_0(\l\,x)=\theta_0(x)$ for every $x\in\bC$ and $\l>0$). The importance of integral cones lies in the fact that if $V$ is a stationary integral $n$-varifold in some open set $U$, $x_0\in\spt\,V$ and $r_j\to0^+$ as $j\to\infty$, then, up to extracting a subsequence of $r_j$, there exists an integral $n$-cone $V_0$ such that
\[
(\iota_{x_0,r_j})_\sharp V \weak V_0\,,
\]
in the varifold convergence (duality with $C^0_c(U\times G_n^{n+1})$), where $\iota_{x,r}(y)=(y-x)/r$ for $x,y\in\R^{n+1}$ and $r>0$; moreover, $V_0$ is stationary in $\R^{n+1}$, and the collection of such limit stationary integral $n$-cones for $V$ at $x_0$ is denoted by
\[
{\rm Tan}(V,x_0)\,.
\]
We recall that if $V_0=\var(\bC,\theta_0)\in{\rm Tan}(V,x_0)$, then
\[
\Theta_V(x_0)=\Theta_{V_0}(0)\ge\Theta_{V_0}(y)\,,\qquad\forall y\in\bC\,.
\]
Correspondingly, the {\bf spine} of the integral $n$-cone $V_0=\var(\bC,\theta_0)$ is defined as
\[
S(V_0)=\Big\{y\in\R^{n+1}:\Theta_{V_0}(y)=\Theta_{V_0}(0)\Big\}\,;
\]
as it turns out, $S(V_0)$ is a linear space in $\R^{n+1}$, and it can actually be characterized as the largest linear space $L$ of $\R^{n+1}$ such that $V_0$ is invariant by translations in $L$, i.e. $(\tau_{v})_{\sharp}V_0=V_0$ for every $v\in L$, where $\tau_v(y)=y+v$ for all $y\in\R^{n+1}$. It is easily seen that if $\dim\,S(V_0)=k\in\{0,...,n\}$ and, without loss of generality, $S(V_0)=\{0\}^{n-k+1}\times\R^k$, then there exist a closed $(n-k)$-cone $\bC_0$ in $\R^{n-k+1}$ and a zero-homogeneous multiplicity function $\phi_0$ on $\bC_0$ such that
\[
\bC=\bC_0 \times \R^k\,, \qquad \theta_0(z,y)=\phi_0(z) \quad \mbox{for $\H^{n-k}$-a.e. $z \in \bC_0$, for every $y \in \R^k$}\,,
\]
and such that $W_0=\var(\bC_0,\phi_0)$ is a stationary integral $(n-k)$-cone in $\R^{n-k+1}$ with
\[
\Theta_{W_0}(0)=\Theta_{V_0}(0)\,, \qquad S(W_0)=\{0\}\,.
\]

The concept of spine leads to defining the notion of {\bf $k$-dimensional stratum} of a stationary integral $n$-varifold $V$ as
\[
\Sc^k(V)=\Big\{x\in\spt\,V: \dim S(V_0)\le k\,,\quad\forall V_0\in{\rm Tan}(V,x)\Big\}\,,
\]
where the classical dimension reduction argument of Federer, see \cite[Appendix A]{SimonLN}, shows that
\begin{equation}
  \label{dimension reduction}
\dim_\H(\Sc^k(V))\le k\qquad \forall\,k=0,...,n\,.
\end{equation}
Moreover, we have the following key result by Naber and Valtorta.

\begin{theorem}[{\cite[Theorem 1.5]{NV_varifolds}}]\label{thm nv} If $V$ is an integral stationary $n$-varifold in an open set $U$ of $\R^{n+1}$, then $\Sc^k(V)$ is countably $k$-rectifiable in $U$ for every $k=0,...,n$.
\end{theorem}

\subsection{Regularity of Almgren minimal sets and proof of Theorem \ref{t:main}}\label{section regularity of Almgren min} We recall that $M$ is an Almgren minimal set in an open set $U\subset\R^{n+1}$ if $M\subset U$ is closed relatively to $U$ and
\begin{equation}
  \label{almgren minimizing set}
\H^n(M\cap B_r(x))\le\H^n(f(M)\cap B_r(x))
\end{equation}
whenever $f$ is a Lipschitz map with $\{f\ne\id\}\cc B_r(x)\cc U$ and $f(B_r(x))\subset B_r(x)$. An immediate consequence of \eqref{almgren minimizing set} is that the multiplicity-one $n$-varifold $V=\var(M,1)$ associated to $M$ is stationary in $U$. The Almgren minimality of $M$ implies that the set of tangent varifolds to $V$ is simpler than it could be in general: indeed, varifold tangent cones to Almgren minimal sets have multiplicity one, and their supports are Almgren minimal cones:

\begin{theorem}[{\cite[Corollary II.2]{taylor76}}]\label{thm taylor1} If $M$ is an Almgren minimal set in $U\subset\R^{n+1}$, $x_0\in M$, and $V_0=\var(\bC,\theta_0)\in{\rm Tan}(\var(M,1),x_0)$, then $\theta_0=1$ on $\bC$, and $\bC$ is an Almgren minimal cone in $\R^{n+1}$.
\end{theorem}

In particular, setting
\begin{eqnarray*}
&&{\rm Tan}(M,x_0)=\Big\{\bC\subset\R^{n+1}:V_0=\var(\bC,1)\in{\rm Tan}(\var(M,1),x_0)\Big\}\,,
\\
&&\mbox{and, correspondingly,  $S(\bC)=S(V_0)$ for every $\bC\in {\rm Tan}(M,x_0)$}\,,
\end{eqnarray*}
we have that
\[
\Sc^k(M)=\Big\{x_0\in M:\dim S(\bC)\le k\,,\quad\forall \bC\in{\rm Tan}(M,x_0)\Big\}
\]
is countably $k$-rectifiable in $\R^{n+1}$ thanks to Theorem \ref{thm nv}.

\begin{remark}[Smoothness criterion]\label{remark smooth}
  {\rm If ${\rm Tan}(M,x_0)$ contains an $n$-dimensional plane, then $M$ is a classical minimal surface in a neighborhood of $M$ as a consequence of Allard's regularity theorem \cite{Allard} and of the fact that $V=\var(M,1)$ is an integral stationary $n$-varifold. As a consequence, the {\bf singular set} $\Sigma$ of $M$ in $U$, defined as the maximal closed subset of $M$ such that $M\setminus\Sigma$ is a smooth minimal surface in $U$, can be characterized as the set of those $x_0\in M$ such that ${\rm Tan}(M,x_0)$ contains no plane.}
\end{remark}

The next important fact is that one can completely characterize Almgren minimal cones in $\R^2$ and $\R^3$:

\begin{theorem}{\cite[Proposition II.3]{taylor76}}\label{thm taylor2}
  If $\bC$ is an Almgren minimal cone in $\R^2$, then, up to rotations, either $\bC=\{0\}\times\R$ or $\bC=\bY^1$.  If $\bC$ in an Almgren minimal cone in $\R^3$, then, up to rotations, either $\bC=\{0\}\times\R^2$, or $\bC=\bY^1\times\R$, or $\bC=\bT^2$.
\end{theorem}

\begin{corollary}\label{corollary taylor 3}
  If $M$ is an Almgren minimal set in $U\subset\R^{n+1}$ and $\bC\in{\rm Tan}(M,x_0)$ for some $x_0\in M$, then, up to rotations, either $\bC=\{0\}\times\R^n$, or $\bC=\bY^1\times\R^{n-1}$, or $\bC=\bT^2\times\R^{n-2}$ or $\dim S(\bC)\le n-3$.
\end{corollary}

\begin{proof}
  One needs to notice that if $\bC=\bC_0\times\R^k$ is an Almgren minimal cone in $\R^{n+1}$, then $\bC_0$ is an Almgren minimal cone in $\R^{n-k+1}$, and combine this fact with Theorem \ref{thm taylor1} and Theorem \ref{thm taylor2}.
\end{proof}

If $M$ is an Almgren minimal set in $U$, $\bC$ is an Almgren minimal cone in $\R^{n+1}$, $\a\in(0,1)$ and $x_0\in M$, then we say that $M$ {\bf admits ambient parametrization of class $C^{1,\a}$ over $\bC$ at $x_0$}, if there exist $r>0$, an open neighborhood $A$ of $x_0$, and a $C^{1,\a}$-diffeomorphism $\Phi:\R^{n+1}\to\R^{n+1}$ such that, $\Phi(0)=x_0$, $\nabla\Phi(0)=\Id$ and
\begin{equation}
  \label{ambient parametrization}
  \Phi\big(B_r\cap\bC\big)=M\cap A\,.
\end{equation}
The main result contained in \cite{taylor76} can be formulated as follows:

\begin{theorem}[{\cite{taylor76}}]\label{thm main taylor}
  If $M$ is an Almgren minimal set in $U\subset\R^3$ and $x_0\in M$, then either $M$ is a classical minimal surface in a neighborhood of $x_0$, or $M$ admits an ambient parametrization of class $C^{1,\a}$ over $\bC$ at $x_0$, where, modulo isometries, $\bC\in\{\bY^1\times\R,\bT^2\}$.
\end{theorem}

\begin{remark}
  {\rm The analysis of Almgren minimal sets in $\R^2$ is noticeably simpler, and it yields the stronger conclusions that $M$ is locally {\it isometric} either to a line or to $\bY^1$: a detailed proof can be easily obtained, for example, by minor adaptations of \cite[Section 30.3]{maggiBOOK}.}
\end{remark}

We are finally in the position to prove Theorem \ref{t:main}.

\begin{proof}[Proof of Theorem \ref{t:main}] Let $(K,E)$ be a generalized minimizer of $\psi(\eps)$. By Theorem \ref{proposition Lipschitz minimizing}, $M=K\setminus\cl(E)$ is an Almgren minimal set in $U=\Om\setminus\cl(E)$. By Corollary \ref{corollary taylor 3} we have that $M=R\cup\Sigma$, where $R$ is a smooth, stable minimal hypersurface in $U\setminus\Sigma$, and $\Sigma$ is a relatively closed subset of $M$ such that if $x_0\in\Sigma$ and $\bC\in{\rm Tan}(M,x_0)$, then either $\bC=\bY^1\times\R^{n-1}$ (modulo isometries), or $\dim\,S(\bC)\le n-2$.

\medskip

If $\bC=\bY^1\times\R^{n-1}\in{\rm Tan}(M,x_0)$, then
\[
V_0=\var(\bC,1)\in{\rm Tan}(V,x_0)
\]
where $V=\var(M,1)$ is an integral stationary $n$-varifold in $U$. By Simon's $Y$-regularity theorem \cite{Simon_cylindrical}, see e.g. \cite[Theorem 4.6]{ColomboEdelenSpolaor} for a handy statement, $M$ can be locally parameterized over $\bY^1\times\R^{n-1}$ near $x_0$, in the sense that there exist $r>0$, an open neighborhood $A$ of $x_0$, and a homeomorphism $\Phi:(\bY^1\times\R^{n-1})\cap B_r\to M\cap A$ with $\Phi(0)=x_0$ and mapping the spine of $\bY^1\times\R^{n-1}$ into $\Sigma\cap A$, such that, denoting by $\{H_i\}_{i=1}^3$ the three $n$-dimensional half-planes whose union gives $\bY^1\times\R^{n-1}$, the restriction of $\Phi$ to $H_i\cap B_r$ is a $C^{1,\a}$-diffeomorphism between hypersurfaces with boundary. An application of Whitney's extension theorem (which is usually mentioned without details in the literature, see e.g. the comments in \cite[Pag. 528]{taylor76} and \cite[Pag. 650]{Simon_cylindrical}; we notice that a simplification of the proof of \cite[Theorem 3.1]{CiLeMaIC1} gives the desired result) allows one to extend $\Phi$ into an ambient parametrization of $M$ over $\bY^1\times\R^{n-1}$ in a neighborhood of $x_0$. However, Theorem \ref{prop:noYx}, excludes the existence of such ambient parametrization. Therefore we conclude that $\bY^1\times\R^{n-1}$ cannot belong modulo isometries to ${\rm Tan}(M,x_0)$ for any $x_0\in\Sigma$. As a consequence, $\dim S(\bC)\le n-2$ for every $\bC\in{\rm Tan}(M,x_0)$, and thus $\Sigma=\Sc^{n-2}(M)$. By Federer's dimensional reduction argument \eqref{dimension reduction}, we conclude in particular that
\[
\H^{n-1}(\Sigma)=0\,.
\]
In summary, $V=\var(M,1)$ is a stationary integral $n$-varifold in $U$, whose regular part is stable thanks to \eqref{minimality KE against diffeos}, and whose singular part is $\H^{n-1}$-negligible. The regularity theory of Schoen, Simon and Wickramasekera \cite{SchoenSimon81,Wic} allows us to conclude then that $\Sigma$ is empty if $1\le n\le 6$, is locally finite in $U$ if $n=7$, and coincides with $\cS^{n-7}(V)$ if $n \ge 8$. In particular, if $n \ge 8$ then $\Sigma$ is countably $(n-7)$-rectifiable in $U$ by Theorem \ref{thm nv}. This completes the proof of the theorem.
\end{proof}

We close with a few technical comments on how the regularity theory for varifolds and Almgren minimal sets has been applied in the above argument.

\begin{remark}\label{rmk reg if n12}
  {\rm In the physical cases $n=1$ and $n=2$, which are clearly the most important ones for the soap film capillarity model, one does not need to use the full power of the regularity theory contained in \cite{Simon_cylindrical,SchoenSimon81,Wic}. Indeed, once $M=K\setminus\cl(E)$ has been shown to be an Almgren minimal set in $\Om\setminus\cl(E)$, Taylor's theorem (i.e., Theorem \ref{thm main taylor} above) shows that if $\Sigma$ is non-empty, then $M$ admits an ambient parametrization over $\bY^1\times\R^{n-1}$ at some of its singular points, thus triggering a contradiction with Theorem \ref{prop:noYx}.}
\end{remark}

\begin{remark}\label{rmk no wic}
  {\rm The following argument allows to use \cite{SchoenSimon81} in place of \cite{Wic} (notice that \cite{Wic} relies on \cite{SchoenSimon81}). Going back to the application of Corollary \ref{corollary taylor 3} to $M=K\setminus\cl(E)$, and after having excluded the existence of $Y$ points thanks to \cite{Simon_cylindrical} and Theorem \ref{prop:noYx}, we are in the position to say that if $\bC\in{\rm Tan}(M,x_0)$, then either $\bC=\bT^2\times\R^{n-2}$ modulo isometries or $\dim\,S(\bC)\le n-3$. In the former case, a direct parametrization argument away from the spine of $\bC$ (in the spirit of \cite[Lemma 4.8]{ColomboEdelenSpolaor}) implies the existence of $Y$-points near $x_0$, and a contradiction with Theorem \ref{prop:noYx}. We thus conclude that $\dim\,S(\bC)\le n-3$ for every $\bC\in{\rm Tan}(M,x_0)$, $x_0\in\Sigma$, and thus that $\H^{n-2}(\Sigma)=0$. By \cite{SchoenSimon81}, an integral stationary $n$-varifold $V$ in $\R^{n+1}$ whose regular part is stable and whose singular set is $\H^{n-2}$-negligible is such that the singular set is empty if $1\le n\le 6$, and coincides with $\cS^{n-7}(V)$ if $n \ge 7$ (and thus it is countably $(n-7)$-recitifiable by Naber-Valtorta).}
\end{remark}

\section{Local finiteness of the Hausdorff measure of the singular set}\label{appendix NV local} In this section we sketch the arguments needed to improve the countable $(n-7)$-rectifiability of $\Sigma$, proved in Theorem \ref{t:main}, into
local finiteness of the $(n-7)$-dimensional Minkowski content, and thus, in particular, into local $\H^{n-7}$-rectifiability; see Remark \ref{remark locally finite}. Towards this goal, we will need to introduce the following notion of {\bf quantitative stratification} of the singular set of a stationary integral varifold.

\medskip

Let $\dist_\var$ be a distance function of the space of $n$-dimensional varifolds in $B_1\subset \R^{n+1}$ which induces the varifold convergence. Let $V$ be a stationary integral $n$-varifold in a ball $B_r(x)\subset \R^{n+1}$ with $x \in \spt(V)$. For any $\delta >0$, we say that $V$ is {\bf $(k,\delta)$-almost symmetric in $B_r(x)$} if there exists a $k$-symmetric integral $n$-cone $V_0$ (i.e. $\dim\,S(V_0)\ge k$) such that
\[
\dist_\var((\iota_{x,r})_\sharp V \mres B_1, V_0\mres B_1) < \delta\,.
\]
For $k \in \{0,\ldots,n\}$ and $\delta>0$, we define the $(k,\de)$-{\bf quantitative stratum} $\cS^k_\de(V)$ by
\begin{equation*}
\begin{split}
\cS^{k}_{\delta}(V) = \Big\{ x \in \spt(V)\, \colon \, & \mbox{$V$ is \emph{not} $(k+1,\delta)$-almost symmetric in $B_r(x)$} \\ & \mbox{for all $r>0$ such that $V$ is stationary in $B_r(x)$} \Big\}\,.
\end{split}
\end{equation*}
We can now recall the following theorem from \cite{NV_varifolds}:

\begin{theorem}[{See \cite[Theorem 1.4]{NV_varifolds}}] \label{thm:nv}
Let $\delta, \Lambda >0$. There exists $C_\delta=C(n,\Lambda,\delta)>0$ such that if $V$ is an integral stationary $n$-varifold in $B_2\subset\R^{n+1}$ with $\|V\|(B_2)\leq \Lambda$ then
\begin{equation} \label{e:minkowski estimate nv}
\Big|  I_r(\cS^k_\delta(V)) \cap B_1  \Big| \leq C_\delta \, r^{n+1-k} \qquad \mbox{for all $0 < r < 1$}\,.
\end{equation}
In particular, $\H^{k}(\cS^k_\delta(V) \cap B_1) \leq C_\delta$. Furthermore, $\cS^k_\de(V)$ is countably $k$-rectifiable.
\end{theorem}

\begin{remark}
The countable $k$-rectifiability of $\cS^k(V)$ claimed in Theorem \ref{thm nv} is in fact a corollary of the countable $k$-rectifiability of the quantitative strata $\cS^k_\de(V)$ together with the fact that
\begin{equation} \label{strata back together}
\cS^k (V)= \bigcup_{\delta >0} \cS^k_\delta(V)\,.
\end{equation}
\end{remark}

We are now in the position to show that, under the assumptions of Theorem \ref{t:main}, if $n\ge 7$, then $\Sigma$ has locally finite $(n-7)$-dimensional Minkowski content, and thus it is locally $\H^{n-7}$-finite. Since we can cover any open set compactly contained in $\Omega \setminus \cl (E)$ by a finite number of balls $B_{3r_*}(x_i)$ such that $B_{r_*}(x_i)$ are pairwise disjoint and $B_{9r_*}(x_i) \subset \Om \setminus \cl (E)$, we can directly focus on obtaining an upper bound on the $(n-7)$-dimensional Minkowski content of $\Sigma$ in $B$ whenever $B$ is an open ball with $3B \subset \Om \setminus \cl (E)$, where $3B$ denotes the concentric ball to $B$ with three times the radius. To this end we claim that
\[
\mbox{$\exists \delta > 0$ such that $\Sigma \cap 2B \subset \cS^{n-7}_\delta(V) \cap 2B$\,.}
\]
Indeed, thanks to Theorem \ref{thm:nv} this claim implies
\[
\Big| I_r(\Sigma) \cap B \Big| \leq C_\de\,r^{8} \qquad \mbox{for all $0<r<{\rm radius}(B)$}\,,
\]
and thus $\H^{n-7}(\Sigma\cap B)\le C_\delta$ for a constant $C_\delta=C(n,\H^n(K \cap 2B),\de)$, from which the local $\H^{n-7}$-finiteness of $\Sigma$ follows. To prove the claim we argue by contradiction and assume the existence of a sequence $\delta_h \to 0^+$ and points $x_h \in \Sigma \cap 2B$ such that $x_h \notin \cS^{n-7}_{\delta_h}(V)$. Assuming that ${\rm radius}(B)=1$ for simplicity, so that $V$ is stationary in $B_1(x_h)$ for every $h$, the definition of quantitative strata then yields a sequence $r_h$ of scales $0 < r_h < 1$ such that $V$ is $(n-6,\delta_h)$-almost symmetric in $B_{r_h}(x_h)$: in other words, there are integral $n$-cones $W_h$ with $\dim S(W_h) \ge n-6$ such that, setting $K_h = (K-x_h)/r_h$ and $V_h = \var(K_h,1)$, we have $\dist_{\var}(V_h \mres B_1, W_h \mres B_1) \leq \delta_h$. Since the weights $\|V_h\|(B_1)$ are uniformly bounded as a consequence of the monotonicity formula, each $V_h$ is stationary in $B_1$, and $\delta_h \to 0^+$, a (not relabeled) subsequence of the varifolds $V_h \mres B_1$ converges, as $h \to \infty$ and in the sense of varifolds, to a stationary integral $n$-varifold which is the restriction to $B_1$ of an Almgren minimal cone $\bC$ in $\R^{n+1}$ with $\dim S(\bC) \ge n-6$. By Remark \ref{remark smooth}, $\bC$ cannot be a plane, as otherwise $K$ would be smooth in a neighborhood of $x_h$ for all sufficiently large $h$, a contradiction to $x_h \in \Sigma$. In particular, $\bC$ is singular at the origin, and since $\dim S(\bC) \ge n-6$ it must be $\H^{n-6}(\Sing(\bC))=\infty$, if $\Sing(\bC)$ denotes the set of singular points of $\bC$. We claim that
\begin{equation} \label{a priori zero}
\H^{n-1}(\Sing(\bC))=0\,.
\end{equation}
If this is true, then we can apply again \cite{Wic} and conclude that $\dim_{\H}(\Sing(\bC))\le n-7$, a contradiction. We prove \eqref{a priori zero} by showing that $\bC$ cannot have points of type $Y$. Otherwise, there would be a point $y \in \bC \cap B_1$ such that, modulo rotations, the (unique) tangent cone $\bC_y$ to $\bC$ at $y$ is $\bY^1 \times \R^{n-1}$. Since varifold convergence of stationary integral varifolds implies Hausdorff convergence of their supports, for every $\delta > 0$ there exists $\sigma \in \left( 0, \dist(y, \pa B_1) \right)$ such that, for all sufficiently large $h$,
\[
\hd(\spt((\iota_{y,\sigma})_\sharp V_h) \cap B_{1}   , (\bY^1 \times \R^{n-1}) \cap B_1  ) \le \delta
\]
where $\hd$ denotes the Hausdorff distance. By Simon's $Y$-regularity theorem, $K_h$ admits an ambient parametrization of class $C^{1,\alpha}$ over $\bY^1 \times \R^{n-1}$ in $B_{\sigma/2}(y)$, and thus, in turn, there is a point in $K$ at which $K$ admits an ambient parametrization of class $C^{1,\alpha}$ over $\bY^1\times\R^{n-1}$, a contradiction to Theorem \ref{prop:noYx}. \qed

\bibliographystyle{is-alpha}
\bibliography{references_mod}

\begin{thebibliography}{DLGM17}

\bibitem[ACV08]{ambrosiocolevilla}
L.~Ambrosio, A.~Colesanti, and E.~Villa.
\newblock Outer {M}inkowski content for some classes of closed sets.
\newblock {\em Math. Ann.}, 342\penalty0 (4):\penalty0 727--748, 2008.

\bibitem[AFP00]{AFP}
L.~Ambrosio, N.~Fusco, and D.~Pallara.
\newblock {\em Functions of bounded variation and free discontinuity problems}.
\newblock Oxford Mathematical Monographs. The Clarendon Press, Oxford
  University Press, New York, 2000.
\newblock xviii+434 pp pp.

\bibitem[All72]{Allard}
W.~K. Allard.
\newblock On the first variation of a varifold.
\newblock {\em Ann. Math.}, 95:\penalty0 417--491, 1972.

\bibitem[Alm76]{Almgren76}
F.~J.~Jr. Almgren.
\newblock Existence and regularity almost everywhere of solutions to elliptic
  variational problems with constraints.
\newblock {\em Mem. Amer. Math. Soc.}, 4\penalty0 (165):\penalty0 viii+199 pp,
  1976.

\bibitem[BR05]{baylerosales}
V.~Bayle and C.~Rosales.
\newblock Some isoperimetric comparison theorems for convex bodies in
  {R}iemannian manifolds.
\newblock {\em Indiana Univ. Math. J.}, 54\penalty0 (5):\penalty0 1371--1394,
  2005.

\bibitem[CES17]{ColomboEdelenSpolaor}
M.~Colombo, N.~Edelen, and L.~Spolaor.
\newblock The singular set of minimal surfaces near polyhedral cones.
\newblock {\em To appear on J. Differential Geom.}, 2017.
\newblock Preprint arXiv:1709.09957.

\bibitem[CLM16]{CiLeMaIC1}
M.~Cicalese, G.~P. Leonardi, and F.~Maggi.
\newblock Improved convergence theorems for bubble clusters {I}. {T}he planar
  case.
\newblock {\em Indiana Univ. Math. J.}, 65\penalty0 (6):\penalty0 1979--2050,
  2016.

\bibitem[DLGM17]{DLGM}
C.~De~Lellis, F.~Ghiraldin, and F.~Maggi.
\newblock A direct approach to {P}lateau's problem.
\newblock {\em J. Eur. Math. Soc. (JEMS)}, 19\penalty0 (8):\penalty0
  2219--2240, 2017.

\bibitem[Fal10]{fall}
M.~M. Fall.
\newblock Area-minimizing regions with small volume in {R}iemannian manifolds
  with boundary.
\newblock {\em Pacific J. Math.}, 244\penalty0 (2):\penalty0 235--260, 2010.

\bibitem[Fed69]{FedererBOOK}
H.~Federer.
\newblock {\em Geometric measure theory}, volume 153 of {\em Die Grundlehren
  der mathematischen Wissenschaften}.
\newblock Springer-Verlag New York Inc., New York, 1969.

\bibitem[GJ86]{gruterjost}
M.~Gr{\"u}ter and J.~Jost.
\newblock Allard type regularity results for varifolds with free boundaries.
\newblock {\em Ann. Scuola Norm. Sup. Pisa Cl. Sci. (4)}, 13\penalty0
  (1):\penalty0 129--169, 1986.

\bibitem[HP16]{harrisonpughACV}
J.~Harrison and H.~Pugh.
\newblock Existence and soap film regularity of solutions to {P}lateau's
  problem.
\newblock {\em Adv. Calc. Var.}, 9\penalty0 (4):\penalty0 357--394, 2016.

\bibitem[HP17]{harrisonpughGENMETH}
J.~Harrison and H.~Pugh.
\newblock General methods of elliptic minimization.
\newblock {\em Calc. Var. Partial Differential Equations}, 56\penalty0
  (4):\penalty0 Art. 123, 25, 2017.

\bibitem[KMS19]{kms}
D.~King, F.~Maggi, and S.~Stuvard.
\newblock Plateau's problem as a singular limit of capillarity problems.
\newblock {\em To appear on Comm. Pure Appl. Math.}, 2019.
\newblock Preprint arXiv:1907.00551.

\bibitem[KMS20]{kms2}
D.~King, F.~Maggi, and S.~Stuvard.
\newblock Collapsing and the convex hull property in a soap film capillarity
  model.
\newblock 2020.
\newblock Preprint arXiv:2002.06273.

\bibitem[Kne55]{kneser}
M.~Kneser.
\newblock Einige {B}emerkungen {\"{u}}ber das {M}inkowskische
  {F}l{\"{a}}chenmass.
\newblock {\em Arch. Math. (Basel)}, 6:\penalty0 382--390, 1955.

\bibitem[KT17]{kagayatone}
T.~Kagaya and Y.~Tonegawa.
\newblock A fixed contact angle condition for varifolds.
\newblock {\em Hiroshima Math. J.}, 47\penalty0 (2):\penalty0 139--153, 2017.

\bibitem[Mag12]{maggiBOOK}
F.~Maggi.
\newblock {\em Sets of finite perimeter and geometric variational problems: an
  introduction to Geometric Measure Theory}, volume 135 of {\em Cambridge
  Studies in Advanced Mathematics}.
\newblock Cambridge University Press, 2012.

\bibitem[MM16]{maggimihaila}
F.~Maggi and C.~Mihaila.
\newblock On the shape of capillarity droplets in a container.
\newblock {\em Calc. Var. Partial Differential Equations}, 55\penalty0
  (5):\penalty0 Art. 122, 42, 2016.

\bibitem[MSS19]{maggiscardicchiostuvard}
F.~Maggi, S.~Stuvard, and A.~Scardicchio.
\newblock Soap films with gravity and almost-minimal surfaces.
\newblock {\em Discrete Contin. Dyn. Syst.}, 39\penalty0 (12):\penalty0
  6877--6912, 2019.

\bibitem[NV15]{NV_varifolds}
A.~Naber and D.~Valtorta.
\newblock The singular structure and regularity of stationary varifolds.
\newblock {\em To appear on J. Eur. Math. Soc. (JEMS)}, 2015.
\newblock Preprint arXiv:1505.03428.

\bibitem[Sim83]{SimonLN}
L.~Simon.
\newblock {\em Lectures on geometric measure theory}, volume~3 of {\em
  Proceedings of the Centre for Mathematical Analysis}.
\newblock Australian National University, Centre for Mathematical Analysis,
  Canberra, 1983.
\newblock vii+272 pp.

\bibitem[Sim93]{Simon_cylindrical}
L.~Simon.
\newblock Cylindrical tangent cones and the singular set of minimal
  submanifolds.
\newblock {\em Journal of Differential Geometry}, 38\penalty0 (3):\penalty0
  585--652, 1993.

\bibitem[SS81]{SchoenSimon81}
R.~Schoen and L.~Simon.
\newblock Regularity of stable minimal hypersurfaces.
\newblock {\em Comm. Pure Appl. Math.}, 34\penalty0 (6):\penalty0 741--797,
  1981.

\bibitem[Tay76]{taylor76}
J.~E. Taylor.
\newblock The structure of singularities in soap-bubble-like and soap-film-like
  minimal surfaces.
\newblock {\em Ann. of Math. (2)}, 103\penalty0 (3):\penalty0 489--539, 1976.

\bibitem[Vil09]{villa}
E.~Villa.
\newblock On the outer {M}inkowski content of sets.
\newblock {\em Ann. Mat. Pura Appl. (4)}, 188\penalty0 (4):\penalty0 619--630,
  2009.

\bibitem[Wic14]{Wic}
N.~Wickramasekera.
\newblock A general regularity theory for stable codimension 1 integral
  varifolds.
\newblock {\em Ann. of Math. (2)}, 179\penalty0 (3):\penalty0 843--1007, 2014.

\end{thebibliography}
\end{document}